\documentclass[a4paper]{amsart}

\usepackage{amsmath,amsfonts,amssymb,amsthm} %AMS packages
\usepackage{mathrsfs} %Script fonts
\usepackage[all, cmtip]{xy} %For commutative diagrams
\usepackage[shortlabels]{enumitem} %Adds extra options to enumerate environment
\usepackage{pict2e} %Extends picture environment to allow more line slopes, circle radii, etc.
\usepackage{caption} %Gives better caption formatting and more caption formatting options
\usepackage[T1]{fontenc} %Extra font characters

\usepackage{mathtools} %Fixes some problems with amsmath and adds some commands
\mathtoolsset{centercolon} %Corrects colons before equals signs

\usepackage{mathabx}

\usepackage{todonotes}
\theoremstyle{definition}
\newtheorem{definition}{Definition}[section]

\theoremstyle{plain}
\newtheorem{lemma}[definition]{Lemma}
\newtheorem{theorem}[definition]{Theorem}
\newtheorem{proposition}[definition]{Proposition}
\newtheorem{corollary}[definition]{Corollary}

\theoremstyle{remark}
\newtheorem{remark}[definition]{Remark}
\newtheorem{example}[definition]{Example}

\newcommand{\bp}{\mathbf{p}}

\newcommand{\bC}{\mathbb{C}}

\newcommand{\bF}{\mathbb{F}}

\newcommand{\bL}{\mathbb{L}}

\newcommand{\bN}{\mathbb{N}}

\newcommand{\bP}{\mathbb{P}}
\newcommand{\bQ}{\mathbb{Q}}
\newcommand{\bR}{\mathbb{R}}

\newcommand{\bZ}{\mathbb{Z}}

\newcommand{\calB}{\mathcal{B}}
\newcommand{\calC}{\mathcal{C}}

\newcommand{\calF}{\mathcal{F}}

\newcommand{\calO}{\mathcal{O}}

\DeclareMathOperator{\Hom}{Hom}

\DeclareMathOperator{\NS}{NS}

\DeclareMathOperator{\rank}{rank}
\DeclareMathOperator{\Span}{Span}

\numberwithin{table}{section}

\begin{document}

\title[Pseudolattices, del Pezzo Surfaces, and Lefschetz Fibrations]{Pseudolattices, del Pezzo Surfaces, and Lefschetz Fibrations}

\author[A. Harder]{Andrew Harder}
\address{Department of Mathematics, Lehigh University, Christmas-Saucon Hall, 14 E. Packer Ave, Bethlehem, PA, 18015, USA}
\email{anh318@lehigh.edu}
\thanks{A. Harder was partially supported by the Simons Collaboration Grant in \emph{Homological Mirror Symmetry}.}

\author[A. Thompson]{Alan Thompson}
\address{Department of Mathematical Sciences, Loughborough University, Loughborough, Leicestershire, LE11 3TU, United Kingdom }
\email{A.M.Thompson@lboro.ac.uk}
\thanks{A. Thompson was partially supported by the Engineering and Physical Sciences Research Council programme grant \emph{Classification, Computation, and Construction: New Methods in Geometry}.}

\begin{abstract} Motivated by the relationship between numerical Grothendieck groups induced by the embedding of a smooth anticanonical elliptic curve into a del Pezzo surface, we define the notion of a quasi del Pezzo homomorphism between pseudolattices and establish its basic properties. The primary aim of the paper is then to prove a classification theorem for quasi del Pezzo homomorphisms, using a pseudolattice variant of the minimal model program. Finally, this result is applied to the classification of a certain class of genus one Lefschetz fibrations over discs.
\end{abstract}

\date{}
\maketitle

\section{Introduction}

The study of pseudolattices was initiated by Kuznetsov \cite{ecslc}, who formalized earlier ideas of Vial \cite{ecnsls}, Perling \cite{caesrs}, de Thanhoffer de Volcsey and Van den Bergh \cite{ameecl4}, and Bondal and Polishchuk \cite{sgtbfbg,hpaamh}. 
Pseudolattices generalize the classical concept of a lattice, by removing the requirement that the bilinear form be symmetric. The motivating example is that of the numerical Grothendieck group $\mathrm{K}_0^{\mathrm{num}}(\mathbf{D}(X))$ associated to the bounded derived category $\mathbf{D}(X)$ of coherent sheaves on a smooth complex projective variety $X$, equipped with the Euler pairing.

Taken generally, pseudolattices form a very broad class, but there are a number of special subclasses which exhibit interesting behaviour. The first, introduced by Kuznetsov \cite{ecslc} and motivated by the properties of $\mathrm{K}_0^{\mathrm{num}}(\mathbf{D}(X))$ when $X$ is a smooth projective complex surface, is the class of \emph{surface-like} pseudolattices; these are characterized by the existence of a special \emph{point-like} element, which behaves like the class of the structure sheaf of a point on $X$. Another class, introduced in Subsection \ref{sec:sphericalhom}, is the class of \emph{$(-1)^n$-Calabi-Yau} pseudolattices, in which the bilinear form is symmetric if $n$ is even and skew-symmetric if $n$ is odd; as suggested by the name, their definition is motivated by the properties of $\mathrm{K}_0^{\mathrm{num}}(\mathbf{D}(X))$ when $X$ is a smooth Calabi-Yau variety of dimension $n$.

Subsection \ref{sec:sphericalhom} also introduces a relative version of the $(-1)^n$-Calabi-Yau condition: this is the notion of a \emph{relative $(-1)^n$-Calabi-Yau homomorphism} between pseudolattices. The motivating example is the derived pull-back $i^*\colon \mathrm{K}_0^{\mathrm{num}}(\mathbf{D}(X)) \to \mathrm{K}_0^{\mathrm{num}}(\mathbf{D}(D))$, where $X$ is a smooth complex projective variety of dimension $n$ containing a smooth anticanonical divisor $D$, and $i\colon D \hookrightarrow X$ is the inclusion.

In this paper, we shall primarily be concerned with a special type of relative $(-1)^n$-Calabi-Yau homomorphism, which we call  \emph{quasi del Pezzo homomorphisms} (see Definition \ref{def:qdp}). The motivating example behind this definition is the derived pull-back $i^*\colon \mathrm{K}_0^{\mathrm{num}}(\mathbf{D}(X)) \to \mathrm{K}_0^{\mathrm{num}}(\mathbf{D}(C))$, where $X$ is a quasi del Pezzo surface (i.e. a rational surface where a general member of the anticanonical linear system is smooth), $C$ is a smooth anticanonical curve, and $i\colon C \hookrightarrow X$ is the inclusion. To a first approximation, one may think of quasi del Pezzo homomorphisms as relatively $(-1)^2$-Calabi-Yau homomorphisms where the source pseudolattice is surface-like, but the full definition is somewhat more technical.

The main result of this paper (Theorem \ref{thm:main}) is a complete classification of quasi del Pezzo homomorphisms. As one may anticipate, this result parallels the well-known classification of quasi del Pezzo surfaces: there are two ``minimal'' quasi del Pezzo homomorphisms, derived from smooth anticanonical divisors in $\bP^2$ and $\bP^1 \times \bP^1$ respectively, to which all other quasi del Pezzo homomorphisms may be reduced by a process of ``contraction''. The argument should therefore be thought of as a kind of minimal model program, refining Kuznetsov's \cite{ecslc} minimal model program for pseudolattices.

After proving this result, the remainder of the paper is dedicated to the application of these results to the classification of genus one Lefschetz fibrations. These are, roughly speaking, oriented, smooth, compact $4$-manifolds $Y$ fibred over (closed) discs, so that a general fibre is a $2$-torus and all singular fibres are obtained from a $2$-torus by contracting a simple closed curve. After choosing a distinguished fibre $C$ over a base point on the boundary of the disc, monodromy acts on $C$ as an element of the mapping class group of a $2$-torus, which is $\mathrm{SL}(2,\bZ)$. The problem is then to determine to what extent the data of the monodromy action and the number of singular fibres determines the Lefschetz fibration.

An application of classical ideas reduces this problem to one about the structure of the relative cohomology group $\mathrm{H}_2(Y,C;\bZ)$. More specifically, after equipping $\mathrm{H}_2(Y,C;\bZ)$ with a bilinear form, which we call the \emph{Seifert pairing} after an analogous notion in singularity theory, it acquires the structure of a pseudolattice, and the classification of genus one Lefschetz fibrations is reduced to a question about the classification of pseudolattices.

In this paper we concern ourselves with the special case where the genus one Lefschetz fibration has $n$ singular fibres and the monodromy action is conjugate to $\begin{psmallmatrix}1 & n-12 \\ 0 & 1\end{psmallmatrix}$; we call such fibrations \emph{quasi Landau-Ginzburg models} after a related notion due to Auroux, Katzarkov, and Orlov \cite{msdpsvccs}. In this case we show (Theorem \ref{thm:Lefschetzisqdp}) that the natural homomorphism $\mathrm{H}_2(Y,C;\bZ) \to \mathrm{H}_1(C,\bZ)$ is in fact a quasi del Pezzo homomorphism of pseudolattices, so we may deduce the classification of quasi Landau-Ginzburg models from our main theorem. In particular, we find that there are two ``fundamental'' quasi Landau-Ginzburg models, with $3$ and $4$ singular fibres respectively, to which all others may be reduced by a process of ``contraction''. In this setting, ``contraction'' intuitively corresponds to the process of identifying a special singular fibre, which may be removed from the quasi Landau-Ginzburg model to give a new, simpler quasi Landau-Ginzburg model that has the same properties as the original.

The problem of classifying genus one Lefschetz fibrations has attracted significant attention in the literature in recent years; we mention here only those papers most closely related to our results. In the case $n \geq 12$, our classification is essentially equivalent to a result of Cadavid and V\'{e}lez on matrix factorizations \cite[Theorem 20]{nfscsfef}; see also Baykur and Kamada \cite[Theorem 7]{cblfsfg} for a similar result stated in terms of Lefschetz fibrations. 

In the remaining cases $n \leq 12$, a number of examples of quasi Landau-Ginzburg models may be constructed by considering open sets in rational elliptic surfaces containing a singular fibre of Kodaira type $\mathrm{I}_{12-n}$; such surfaces were classified by Persson in \cite{ckfres}. These examples have formed the basis for a large volume of physics literature on \emph{string junctions} (see, for example, \cite{sjalar,mwlsj}), and were used by Auroux \cite{fsseisf} to demonstrate that there are two inequivalent genus one Lefschetz fibrations in the case $n = 4$. The question of whether these examples comprise \emph{all} quasi Landau-Ginzburg models with $n \leq 12$ formed the original motivation for writing this paper; it follows from our results that they do.

We also note that there is also a somewhat similar classification result in the case where the genus one Lefschetz fibration has $n$ singular fibres and the monodromy action has the related form $\begin{psmallmatrix}-1 & n-6 \\ 0 & -1\end{psmallmatrix}$; this was proved by Baykur and Kamada \cite[Theorem 7]{cblfsfg} for $n \geq 6$ and Golla and Lisca \cite[Theorem 3.5]{sfctb} for $n \geq 2$. One may use this to deduce a much weaker form of our result by adding two additional singular fibres in a controlled way. 
%Do we want to say more about this? Maybe mention that they do not make any statements about fibration structures on the result? By Auroux Section 5, it is unclear whether a single contact filling may admit more than one fibration.

Finally, we remark that the fact that the classification of quasi Landau-Ginzburg models parallels the classification of quasi del Pezzo surfaces may seem surprising, but in the context of mirror symmetry this result is very natural. Indeed, if we postulate a mirror relationship between quasi del Pezzo surfaces and quasi Landau-Ginzburg models, so that the quasi del Pezzo surface $X$ is mirror to a quasi Landau-Ginzburg model $Y$, then homological mirror symmetry predicts that the derived category of coherent sheaves $\mathbf{D}(X)$ on $X$ should be derived equivalent to the derived Fukaya-Seidel category $\mathbf{F}(Y)$ of $Y$.  Morover, if $Y$ has chosen fibre $C$, it follows from the discussion in \cite[Section 6.1]{fscslf} that  $\mathrm{K}_0^{\mathrm{num}}(\mathbf{F}(Y))$, with its Euler pairing, is isomorphic as a pseudolattice to $\mathrm{H}_2(Y,C;\bZ)$ with the Seifert pairing. So we should expect $\mathrm{K}_0^{\mathrm{num}}(\mathbf{D}(X))$ to be isomorphic to $\mathrm{K}_0^{\mathrm{num}}(\mathbf{F}(Y))$, and the same classification result should hold in both cases. Further discussion of the relationship between our work and mirror symmetry may be found in Section \ref{sec:mirrors}.
\medskip

The structure of this paper is as follows. In Section \ref{sec:pseudolattices}, we introduce the basic pseudolattice machinery that will be used throughout the paper, including the definitions and fundamental properties of $(-1)^n$-Calabi-Yau pseudolattices and relative $(-1)^n$-Calabi-Yau homomorphisms. Much of the elementary material in this section is originally due to Kuznetsov \cite{ecslc}, whilst the material on pseudolattice homomorphisms is inspired by the work of Anno and Logvinenko on spherical functors \cite{sdgf}.

The main section of the paper is Section \ref{sec:qdp}, where we introduce the central notion of a quasi del Pezzo homomorphism (Definition \ref{def:qdp}) and prove some of its basic properties. This section culminates in the proof of the classification theorem for quasi del Pezzo homomorphisms (Theorem \ref{thm:main}). To prove this theorem, we run a carefully controlled, relative variant of Kuznetsov's \cite{ecslc} minimal model program for pseudolattices, that allows us to guarantee that our homomorphism remains quasi del Pezzo after each contraction step.

In Section \ref{sec:fibrations}, we apply these ideas to the classification of genus one Lefschetz fibrations over discs. After recalling relevant background material we introduce the notion of a quasi Landau-Ginzburg model (Definition \ref{def:lg}). Most of the remainder of the section is then dedicated to showing that the classification of quasi Landau-Ginzburg models may be reduced to that of quasi del Pezzo homomorphisms (Theorem \ref{thm:Lefschetzisqdp}), from which we may deduce the classification of quasi Landau-Ginzburg models (Theorem \ref{thm:lgclass} and Corollary \ref{cor:lgclass}).

Finally, Section \ref{sec:mirrors} contains a brief discussion of the relationship between this work and mirror symmetry. The main result here is Theorem \ref{thm:mirror}, which should be thought of as the shadow of a homological mirror symmetry statement on the numerical Grothendieck groups.
\medskip

\noindent \textbf{Acknowledgments.} The idea for this paper arose following discussions between Charles Doran and the authors during a visit to the Harvard Center of Mathematical Sciences and Applications (CMSA) in April 2018; the authors would like to thank the CMSA for their kind hospitality. The authors would also like to thank Ivan Smith, for some insightful comments on Lefschetz fibrations.

\section{Pseudolattices and Spherical Homomorphisms}\label{sec:pseudolattices}

We begin by discussing the theory of pseudolattices; this theory will form the backbone upon which the rest of the paper is supported.  This theory is largely based on the work of Kuznetsov \cite{ecslc}, who formalized earlier ideas of Vial \cite{ecnsls}, Perling \cite{caesrs}, de Thanhoffer de Volcsey and Van den Bergh \cite{ameecl4}, and Bondal and Polishchuk \cite{sgtbfbg,hpaamh}; Subsections \ref{sec:pseudolatticegeneralities} and \ref{sec:surfacelike} are essentially a recap of Kuznetsov's ideas. 

\subsection{Generalities on pseudolattices} \label{sec:pseudolatticegeneralities}

We begin with the central definition.

\begin{definition} \cite[Definition 2.1]{ecslc}
A \emph{pseudolattice} $(\mathrm{G},\langle \cdot, \cdot \rangle_\mathrm{G})$ is a finitely generated free abelian group $\mathrm{G}$ equipped with a (not necessarily symmetric) nondegenerate bilinear form
\[
\langle \cdot, \cdot \rangle_\mathrm{G} \colon \mathrm{G} \times \mathrm{G} \rightarrow \mathbb{Z}.
\]
A pseudolattice is called \emph{unimodular} if the form $\langle \cdot, \cdot \rangle_\mathrm{G}$ induces an isomorphism between $\mathrm{G}$ and $\mathrm{Hom}_\mathbb{Z}(\mathrm{G},\mathbb{Z})$. To simplify notation, we will usually refer to $(\mathrm{G},\langle \cdot, \cdot \rangle_\mathrm{G})$ by its underlying group $\mathrm{G}$.
\end{definition}

Many pseudolattices $\mathrm{G}$ come equipped an automorphism $S_\mathrm{G}$ of $\mathrm{G}$, called the \emph{Serre operator}. 

\begin{definition} \cite[Section 2.1]{ecslc}
Let $\mathrm{G}$ be a pseudolattice. A \emph{Serre operator} $S_\mathrm{G}$ is an automorphism of $\mathrm{G}$ so that 
\[
\langle u_1 , u_2 \rangle_\mathrm{G} = \langle u_2, S_\mathrm{G}(u_1) \rangle_\mathrm{G}
\]
for all $u_1,u_2 \in \mathrm{G}$. If a Serre operator exists, then it is unique.
\end{definition}

If we let $\chi$ denote the Gram matrix of $\langle \cdot , \cdot \rangle_\mathrm{G}$ and treat $u$ as a column vector, then it is a simple exercise in linear algebra to show that the Serre operator $S_\mathrm{G}$ is given by the matrix product
\[
S_\mathrm{G}(u) = \chi^{-1} \chi^T u.
\]
If $\mathrm{G}$ is unimodular then $\chi^{-1}$ is an integral matrix, so every unimodular pseudolattice has a Serre operator $S_\mathrm{G}$.

\begin{definition} \cite[Definition 2.2]{ecslc}
An element $e \in\mathrm{G}$ is called \emph{exceptional} if $\langle e,e\rangle_{\mathrm{G}} = 1$. A sequence of elements $(e_1,\ldots,e_n)$ is called an \emph{exceptional sequence} if $\langle e_i,e_j \rangle_{\mathrm{G}} = 0$ for all $i > j$. An \emph{exceptional basis} for $\mathrm{G}$ is an exceptional sequence whose elements form a basis for $\mathrm{G}$.
\end{definition}

By \cite[Lemma 2.3]{ecslc}, if $\mathrm{G}$ has an exceptional sequence of length $\rank(\mathrm{G})$, then this sequence forms an exceptional basis and $\mathrm{G}$ is unimodular. The motivating example behind these definitions is as follows. 

\begin{example} \cite[Section 2]{ecslc} \label{ex:Grothendieck} Given a saturated, $k$-linear triangulated category $\calC$ (where $k$ denotes any field), one may define the \emph{Grothendieck group} $\mathrm{K}_0(\calC)$ of $\calC$, which comes equipped with the \emph{Euler bilinear form} $\langle \cdot, \cdot\rangle \colon \mathrm{K}_0(\calC) \times \mathrm{K}_0(\calC) \to \bZ$, defined by
\[\langle\calF_1,\calF_2\rangle  := \sum_i(-1)^i\dim \Hom(\calF_1,\calF_2[i]).\]
The \emph{numerical Grothendieck group} $\mathrm{K}_0^{\mathrm{num}}(\calC)$ is then defined to be the quotient of $\mathrm{K}_0(\calC)$ by the kernel of the Euler form. 

Kuznetsov shows that, when equipped with the Euler form, $\mathrm{K}_0^{\mathrm{num}}(\calC)$ forms a pseudolattice. Moreover, the Serre functor on $\calC$ induces a Serre operator on $\mathrm{K}_0^{\mathrm{num}}(\calC)$, the classes of exceptional objects in $\calC$ are exceptional elements in  $\mathrm{K}_0^{\mathrm{num}}(\calC)$, any exceptional collection in $\calC$ gives an exceptional sequence in  $\mathrm{K}_0^{\mathrm{num}}(\calC)$, and any full exceptional collection in $\calC$ gives an exceptional basis of  $\mathrm{K}_0^{\mathrm{num}}(\calC)$ (although we note that the converses of these statements may not be true). \end{example}

Given an exceptional sequence, we may obtain another by the process of \emph{mutation}.

\begin{definition} Assume that $e \in \mathrm{G}$ is exceptional. The \emph{left} and \emph{right mutations} with respect to $e$ are endomorphisms of $\mathrm{G}$ defined, respectively, by
\[\bL_e(u) := u - \langle e,u \rangle_{\mathrm{G}}e \quad \text{and} \quad \bR_e(u) := u - \langle u,e \rangle_{\mathrm{G}}e.\]
\end{definition}

Given an exceptional sequence $e_{\bullet} = (e_1,\ldots,e_n)$, we may obtain new exceptional sequences by
\begin{align*}
\bL_{i,i+1}(e_{\bullet}) &:= (e_1,\ldots,e_{i-1},\bL_{e_i}(e_{i+1}),e_i,e_{i+2},\ldots,e_n),\\
\bR_{i,i+1}(e_{\bullet}) &:= (e_1,\ldots,e_{i-1},e_{i+1},\bR_{e_{i+1}}(e_i),e_{i+2},\ldots,e_n),
\end{align*}
and these two operations are mutually inverse.

\subsection{Surface-like pseudolattices}\label{sec:surfacelike}

There is an important subclass of pseudolattices, called the \emph{surface-like} pseudolattices. Many of the pseudolattices considered in this paper will turn out to be surface-like.

\begin{definition}\label{def:surfacelike} \cite[Definition 3.1]{ecslc}
A pseudolattice $\mathrm{G}$ is \emph{surface-like} if there is a primitive element $\mathbf{p} \in\mathrm{G}$ such that
\begin{enumerate}
\item $\langle \mathbf{p},\mathbf{p} \rangle_\mathrm{G} = 0$.
\item $\langle \mathbf{p} ,u \rangle_\mathrm{G} = \langle u , \mathbf{p}\rangle_\mathrm{G}$ for all $u \in \mathrm{G}$.
\item If $u_1,u_2 \in \mathrm{G}$ satisfy $\langle u_1,\mathbf{p} \rangle_\mathrm{G}= \langle u_2 ,\mathbf{p} \rangle_\mathrm{G} = 0$, then $\langle u_1,u_2 \rangle_\mathrm{G} = \langle u_2,u_1 \rangle_\mathrm{G}$.
\end{enumerate}
An element $\bp$ with the above properties is called a \emph{point-like} element in $\mathrm{G}$.
\end{definition}

Using the point-like element $\bp$, one may define a rank function on $\mathrm{G}$. Note that this definition, and the others appearing in the remainder of this subsection, depend upon the choice of point-like element $\bp$; we will see in Example \ref{ex:dx} that a surface-like pseudolattice may contain more than one point-like element.

\begin{definition}\cite[Section 3.2]{ecslc}
The \emph{rank} of $u \in \mathrm{G}$ is the integer
\[
\rank(u) := \langle  u, \mathbf{p} \rangle_\mathrm{G} = \langle \bp, u \rangle_{\mathrm{G}}.
\]
\end{definition}

By \cite[Lemma 3.10]{ecslc} the rank function fits into a complex
\[\bZ \stackrel{\bp}{\longrightarrow} G \xrightarrow{\rank} \bZ,\]
where the first map is always injective, and the middle cohomology ${\mathbf{p}^{\perp}} / \mathbf{p}$ of this complex is a finitely generated free abelian group. \cite[Lemma 3.11]{ecslc} shows that the bilinear form on $\mathrm{G}$ induces a bilinear form on ${\mathbf{p}^{\perp}} / \mathbf{p}$, making it into a lattice (in the traditional sense of a finitely-generated free abelian group equipped with an integral symmetric bilinear form). This lattice is called the \emph{N\'eron-Severi lattice} of $\mathrm{G}$; its formal definition is as follows.

\begin{definition}
Let $\mathrm{G}$ be a surface-like pseudolattice. Define the \emph{N\'eron-Severi group} of $\mathrm{G}$ to be the group
\[
\NS(\mathrm{G}) := {\mathbf{p}^{\perp}} / \mathbf{p};
\]
it is a finitely generated free abelian group. The N\'eron-Severi group is equipped with a nondegenerate integral symmetric bilinear form $q(\cdot, \cdot)$, defined as follows: if $u_1,u_2 \in {\mathbf{p}}^{\perp}$ and $[u_1], [u_2]$ are their classes in $\NS(\mathrm{G})$, then
\[
q\left([u_1], [u_2]\right) := -\langle u_1, u_2 \rangle_\mathrm{G};
\]
it is not difficult to show that this definition is independent of the choice of representatives for $[u_1]$, $[u_2]$. The pair $(\NS(\mathrm{G}), q(\cdot,\cdot))$ is called the \emph{N\'eron-Severi lattice} $\mathrm{G}$. As before, to simplify notation, we will usually refer to $(\NS(\mathrm{G}), q(\cdot,\cdot))$ by its underlying group $\NS(\mathrm{G})$. 
\end{definition}

Kuznetsov \cite[Lemma 3.11]{ecslc} shows that $\mathrm{G}$ is unimodular if and only if  $\NS(\mathrm{G})$ is unimodular and the rank function is surjective.  Using the rank function, one may define a map $\lambda\colon \bigwedge^2 \mathrm{G}\to \mathrm{NS}(\mathrm{G})$ by
\[
\lambda( u_1 \wedge u_2) := [\rank(u_1) u_2 - \rank(u_2) u_1].
\]
This map may be used to define a distinguished class $K_\mathrm{G}$  in $\NS(\mathrm{G}) \otimes \bQ$ , called the \emph{canonical class} of $\mathrm{G}$. Existence and uniqueness of the canonical class was proved by Kuznetsov in \cite[Lemma 3.12]{ecslc}.

\begin{definition}
The \emph{canonical class of $\mathrm{G}$} is the unique class $K_\mathrm{G} \in \mathrm{NS}(\mathrm{G}) \otimes \bQ$ such that for every $u_1,u_2 \in \mathrm{G}$,
\[
\langle u_1,u_2 \rangle_\mathrm{G} - \langle u_2,u_1 \rangle_\mathrm{G} = - q(K_\mathrm{G},\lambda(u_1 \wedge u_2) ).
\]
\end{definition}

By \cite[Lemma 3.12]{ecslc}, if $\mathrm{G}$ is unimodular, then $K_\mathrm{G}$ is integral (i.e. $K_{\mathrm{G}} \in \mathrm{NS}(\mathrm{G})$). Using $K_{\mathrm{G}}$, we may define an important invariant of a  surface-like pseudolattice.

\begin{definition}\cite[Definition 5.3]{ecslc} Let $\mathrm{G}$ be a surface-like pseudolattice with canonical class $K_{\mathrm{G}}$. The \emph{defect of $\mathrm{G}$} is defined to be
\[\delta(\mathrm{G}) := q(K_{\mathrm{G}},K_{\mathrm{G}}) + \rank(\NS(\mathrm{G})) - 10.\]
\end{definition}

Note that $\delta(\mathrm{G})$ is an integer if $\mathrm{G}$ is unimodular.  The motivating example behind these definitions is as follows. 

\begin{example} \label{ex:dx} \cite[Example 3.5]{ecslc} Suppose that $X$ is a smooth complex projective surface and  let $\mathbf{D}(X)$ denote the bounded derived category of coherent sheaves on $X$. Then $\mathrm{K}_0^{\mathrm{num}}(\mathbf{D}(X)))$ is a surface-like pseudolattice, with point-like element $\bp_X$ given by the class of the structure sheaf of a point. In this case, the N\'{e}ron-Severi lattice $\NS(\mathrm{K}_0^{\mathrm{num}}(\mathbf{D}(X)))$ is the usual N\'{e}ron-Severi group of $X$ with its intersection form, and the canonical class is the canonical class of $X$. 

Furthermore, Kuznetsov \cite[Lemma 5.5]{ecslc} has shown that if the geometric genus and irregularity of $X$ both vanish (this happens, for instance, when $X$ is a rational surface), then the defect $\delta(\mathrm{K}_0^{\mathrm{num}}(\mathbf{D}(X))) = 0$.

Finally, it follows from \cite[Example 3.5]{ecslc} that the point-like element $\bp_X$ given above is unique unless the canonical class of $X$ satisfies $K_X^2 = 0$. In this latter case, a nice example of a ``non-standard'' point-like element in $\mathrm{K}_0^{\mathrm{num}}(\mathbf{D}(X)))$ is described in \cite[Example 3.6]{ecslc}.
\end{example}

\subsection{Spherical homomorphisms} \label{sec:sphericalhom}

We next develop the theory of homomorphisms between pseudolattices. The ideas in the next few subsections are inspired by the work of Anno and Logvinenko \cite{sdgf} on spherical functors. Kapranov and Schechtman \cite{ps} noticed that such spherical functors may be thought of as categorical analogues of perverse sheaves on a disc, as studied by Galligo, Granger and Maisonobe \cite{dmfpsscn}. Under this analogy, spherical homomorphisms of pseudolattices (see Definition \ref{def:spherical}) correspond to the objects Kapranov and Schechtman \cite{psgs} call polarized perverse sheaves. We begin with a definition. 

\begin{definition}\label{def:CYpseudolattice}
We will say that a pseudolattice $\mathrm{G}$ is \emph{$(-1)^n$-Calabi-Yau} (often abbreviated to \emph{$(-1)^n$-CY}) if 
\[
\langle u_1, u_2 \rangle_\mathrm{G} = (-1)^n\langle u_2, u_1 \rangle_\mathrm{G}
\]
for all $u_1,u_2 \in \mathrm{G}$. In other words, $\mathrm{G}$ is a $(-1)^n$-Calabi-Yau pseudolattice if $\mathrm{G}$ has Serre operator $S_\mathrm{G} = (-1)^n\mathrm{id}_\mathrm{G}$. Note that this definition only depends upon the parity of $n$, so most of the time we will treat $n$ as an element of $\bZ/2\bZ$.
\end{definition}

The motivating example behind this definition is as follows.

\begin{example}\label{ex:CalabiYau} Let $\calC$ be a smooth and proper Calabi-Yau category of dimension $n$ (i.e. its Serre functor is the shift by $n$). Then the Euler pairing on $\mathrm{K}_0^{\mathrm{num}}(\calC)$ is symmetric if $n$ is even and skew-symmetric if $n$ is odd, so $\mathrm{K}_0^{\mathrm{num}}(\calC)$ is a $(-1)^n$-CY pseudolattice.
\end{example}

The next example is important enough that we state it as a definition.

\begin{definition} \label{def:E} The \emph{elliptic curve pseudolattice}, henceforth denoted by $\mathrm{E}$, is the unimodular $(-1)^1$-Calabi-Yau pseudolattice with underlying group $\mathbb{Z}^2$ generated by primitive elements $a,b$, and bilinear form defined by
\[
\langle a,b \rangle_\mathrm{E} = -1, \quad\quad \langle b,a \rangle_\mathrm{E} = 1, \quad\quad \langle a,a \rangle_\mathrm{E} = \langle b,b \rangle_\mathrm{E} = 0,
\]
extended to $\bZ^2$ by linearity.
\end{definition}

\begin{example} \label{ex:elliptic}
The pseudolattice $\mathrm{E}$ arises from an elliptic curve $C$ in two ways.
\begin{enumerate}
\item The numerical Grothendieck group $\mathrm{K}_0^{\mathrm{num}}(\mathbf{D}(C))$, with its Euler pairing, is isomorphic to $\mathrm{E}$. The element $a$ may be identified with the class of the structure sheaf $\calO_p$ of a point $p \in C$, and $b$ may be identified with the class of $\calO_C$.
\item The first integral homology $\mathrm{H}_1(C,\bZ)$, with its usual intersection form, is isomorphic to $\mathrm{E}$. This may also be interpreted as the numerical Grothendieck group associated to a triangulated category, in this case the \emph{derived Fukaya category} $\mathbf{F}(C)$ associated to $C$.
\end{enumerate}
The first definition should be thought of as ``algebraic'' and the second as ``symplectic''. Homological mirror symmetry for elliptic curves identifies the two sides; we will discuss this further in Section \ref{sec:mirrors}.
\end{example}

We now come to the main definition in this subsection.

\begin{definition} \label{def:spherical}
Let $\mathrm{G}$ and $\mathrm{H}$ be pseudolattices. A \emph{spherical homomorphism} from $\mathrm{G}$ to $\mathrm{H}$ is a homomorphism of abelian groups $f\colon\mathrm{G} \to \mathrm{H}$ with the following properties:
\begin{enumerate}
\item The homomorphism $f$ has a right adjoint $r\colon \mathrm{H} \to \mathrm{G}$, in the sense that 
\[
\langle f(u), v \rangle_{\mathrm{H}} = \langle u, r(v) \rangle_{\mathrm{G}}
\]
for any $u \in \mathrm{G}$ and $v \in \mathrm{H}$. 
\item The \emph{twist} and \emph{cotwist} endomorphisms, which are defined as
\[
{T}_f := \mathrm{id}_{\mathrm{H}} - fr, \quad \quad {C}_f := \mathrm{id}_{\mathrm{G}} - rf,
\]
respectively, are invertible. In fact, if ${T}_f$ is invertible with inverse $T_f^{-1}$, then ${C}_f$ is invertible with inverse $C_f^{-1} = \mathrm{id}_{\mathrm{G}} + rT_f^{-1}f$, and vice versa.
\end{enumerate}
\end{definition}

\begin{remark} Since the bilinear forms on $\mathrm{G}$ and $\mathrm{H}$ are non-degenerate, it is easy to show that if a right adjoint for $f$ exists, then it must be unique. Moreover, a simple exercise in linear algebra shows that if $\mathrm{G}$ is unimodular, then a right adjoint for $f$ always exists.
\end{remark}

Spherical homomorphisms allow us to define a relative notion of the $(-1)^n$-CY property.

\begin{definition}
We say that a spherical homomorphism $f \colon \mathrm{G} \rightarrow \mathrm{H}$ is \emph{relative $(-1)^n$-Calabi-Yau} (often abbreviated to \emph{relative $(-1)^n$-CY}) if $\mathrm{G}$ has a Serre operator $S_{\mathrm{G}}$ and ${C}_f = (-1)^nS_{\mathrm{G}}$. As in Definition \ref{def:CYpseudolattice}, this definition depends only upon the parity of $n$, so most of the time we will treat $n$ as an element of $\bZ/2\bZ$.
\end{definition}

The motivating example behind this definition is as follows.

\begin{example} \label{ex:relativeCY} Let $X$ be a smooth complex projective variety of dimension $n$ which contains a smooth anticanonical divisor $D$. Then we have a natural homomorphism $i^*\colon \mathrm{K}_0^{\mathrm{num}}(\mathbf{D}(X)) \to \mathrm{K}_0^{\mathrm{num}}(\mathbf{D}(D))$, given by the derived pull-back under the inclusion $i\colon D \hookrightarrow X$. This homomorphism has right adjoint $i_*$, given by derived push-forward. It is easy to check that $i^*$ is spherical.

The Serre functor on $\mathbf{D}(X)$ is given by $S_{\mathbf{D}(X)}=(- \otimes \omega_X)[n]$, where $\omega_X$ is the canonical sheaf of $X$. Since $\omega_X \cong \calO_X(-D)$, it is easy to compute that the Serre operator on $\mathrm{K}_0^{\mathrm{num}}(\mathbf{D}(X))$ may be written as 
\[S_{\mathrm{K}_0^{\mathrm{num}}(\mathbf{D}(X))}(\{\calF\}) = (-1)^n\{\calF(-D)\} = (-1)^n(\{\calF\} - \{\calF \otimes i_*(\calO_D)\}),\] 
where $\{\calF\}$ denotes the class of $\calF$ in $\mathrm{K}_0^{\mathrm{num}}(\mathbf{D}(X))$. Finally, the projection formula gives that $\calF \otimes i_*(\calO_D) = i_*i^*(\calF)$, so we obtain $S_{\mathrm{K}_0^{\mathrm{num}}(\mathbf{D}(X))} = (-1)^n(\mathrm{id} - i_*i^*) = (-1)^nC_{i^*}$. Thus $i^*$ is relative $(-1)^n$-CY.
\end{example}

Another example, which will be more important for us, is closely related to spherical objects in category theory.

\begin{example}\label{ex:sphob}
Let $\mathrm{H}$ be a $(-1)^n$-CY pseudolattice and let $v \in \mathrm{H}$ be a primitive vector so that 
\[\langle v, v\rangle_\mathrm{H} = \begin{cases}2 & \text{if } n \equiv 0 \bmod 2 \\   0 & \text{if } n \equiv 1 \bmod 2.\end{cases} \] 
Let $\mathrm{Z}(v)$ be the unimodular $(-1)^0$-CY pseudolattice generated by a single element $z$, with $\langle z, z \rangle_{\mathrm{Z}(v)}  =1$. Note that $(z)$ is an exceptional basis for $\mathrm{Z}(v)$.

There is a spherical homomorphism $\zeta\colon \mathrm{Z}(v) \rightarrow \mathrm{H}$ sending $z$ to $v$. The right adjoint of this homomorphism is
\begin{align*}
\rho\colon \mathrm{H} &\longrightarrow \mathrm{Z}(v) \\ w &\longmapsto \langle v, w \rangle_\mathrm{H} z
\end{align*}
This is easy to check: $\langle \zeta(z), w \rangle_\mathrm{H} = \langle v, w \rangle_\mathrm{H} = \langle z, \rho(w) \rangle_{\mathrm{Z}(v)}$. 

The cotwist of $\zeta$ acts on $m\cdot z \in \mathrm{Z}(v)$ by
\[
C_{\zeta}(m\cdot z) := (\mathrm{id}_{\mathrm{Z}} - \rho\zeta)(m\cdot z) \mapsto \left(m - m \langle v, v \rangle_{\mathrm{H}}\right)\cdot z = (-1)^{n+1}m\cdot z,
\]
so $C_{\zeta} = (-1)^{n+1}\mathrm{id}_{\mathrm{Z}(v)}$. Since $\mathrm{Z}(v)$ is $(-1)^0$-CY, it follows that ${C}_{\zeta} = (-1)^{n+1} S_{\mathrm{Z}(v)}$, hence $\zeta$ is relative $(-1)^{n+1}$-CY.

The twist of $\zeta$ acts on $w \in \mathrm{H}$ by 
\[
T_{\zeta}(w) := (\mathrm{id}_\mathrm{H}  - \zeta\rho)(w) = w - \langle v, w \rangle_\mathrm{H} v,
\]
which is the reflection associated to $v$. If $\mathrm{E}$ is the elliptic curve pseudolattice, identified with $\mathrm{H}_1(C,\bZ)$ for an elliptic curve $C$, and $v \in \mathrm{H}_1(C,\bZ)$ is the class of a simple oriented cycle, then $T_{\zeta}$ is the action on $\mathrm{H}_1(C,\bZ)$ induced by the Dehn twist associated to $v$.
\end{example}

We conclude this subsection with a useful lemma about relative $(-1)^n$-CY homomorphisms, that will be used repeatedly in the rest of the paper.

\begin{lemma}\label{lem:useful} Let $f\colon \mathrm{G} \to \mathrm{H}$ be a relative $(-1)^n$-CY spherical homomorphism. Then for any $u_1,u_2 \in \mathrm{G}$,
\[\langle f(u_1),f(u_2) \rangle_{\mathrm{H}} = \langle u_1,u_2 \rangle_{\mathrm{G}} + (-1)^{n+1} \langle u_2,u_1 \rangle_{\mathrm{G}}.\]
\end{lemma}
\begin{proof} This is a simple computation. Let $r$ denote the right adjoint of $f$; then we have
\begin{align*}
\langle f(u_1), f(u_2) \rangle_\mathrm{H} &= \langle u_1, rf(u_2) \rangle_\mathrm{G} \\
&= \langle u_1, u_2\rangle_\mathrm{G} -  \langle u_1, C_f(u_2) \rangle_\mathrm{G} \\
&= \langle u_1 ,u_2 \rangle_\mathrm{G} - \langle u_1, (-1)^n S_\mathrm{G}(u_2) \rangle_\mathrm{G} \\
&= \langle u_1, u_2 \rangle_\mathrm{G} + (-1)^{n+1} \langle u_2 , u_1 \rangle_\mathrm{G}.
\end{align*}
\end{proof}

\subsection{Gluing pseudolattices along spherical homomorphisms} \label{sec:gluing} We may form new pseudolattices by gluing old ones along spherical homomorphisms.

\begin{definition}\label{def:gluing}
Let $\mathrm{G}_1, \mathrm{G}_2$ and  $\mathrm{H}$ be pseudolattices. Assume that 
\[
f_1 \colon \mathrm{G}_1 \longrightarrow \mathrm{H}, \quad\quad f_2 \colon \mathrm{G}_2 \longrightarrow \mathrm{H}
\]
are homomorphisms of the underlying groups. Define $\mathrm{G}_1 \oright_{\mathrm{H}} \mathrm{G}_2$ to be the pseudolattice with underlying group $\mathrm{G}_1 \oplus \mathrm{G}_2$ and bilinear form $\langle \cdot , \cdot \rangle_{\mathrm{G}_1 \oright_{\mathrm{H}} \mathrm{G}_2}$ given as follows: for $u_i \in \mathrm{G}_i$ and $v_j \in \mathrm{G}_j$, define
\[
\langle u_i, v_j \rangle_{\mathrm{G}_1 \oright_{\mathrm{H}}\mathrm{G}_2}
=\begin{cases}
\langle u_i, v_j \rangle_{\mathrm{G}_i} & \text{if }  i = j\\
\langle f_i(u_i), f_j(v_j) \rangle_\mathrm{H} & \text{if } i=1,\ j = 2 \\
0 & \text{if }  i=2,\ j=1 
\end{cases}
\]
and extend to $\mathrm{G}_1 \oplus \mathrm{G}_2$ by linearity.  When there is no chance of confusion, we will use the notation $\mathrm{G}_1 \oright \mathrm{G}_2$ instead of $\mathrm{G}_1 \oright_{\mathrm{H}} \mathrm{G}_2$. 
\end{definition}

It is easy to see that $\mathrm{G}_1 \oright \mathrm{G}_2$ is unimodular if $\mathrm{G}_1$ and $\mathrm{G}_2$ are, and that if $e_{1\bullet}$ and $e_{2\bullet}$ are exceptional bases of $\mathrm{G}_1$ and $\mathrm{G}_2$, then $(e_{1\bullet},e_{2\bullet})$ is an exceptional basis for $\mathrm{G}_1 \oright \mathrm{G}_2$. Note also that this definition does not require the $f_i$ to be spherical homomorphisms; however, as the next proposition shows, if we do assume that the $f_i$ are spherical, then the natural homomorphism $\mathrm{G}_1 \oright \mathrm{G}_2 \rightarrow \mathrm{H}$ is also spherical.

\begin{proposition}\label{prop:glueCY}
Given three pseudolattices $\mathrm{G}_1,\mathrm{G}_2$ and $\mathrm{H}$ and spherical homomorphisms $f_i\colon \mathrm{G}_i \rightarrow \mathrm{H}$, there is a natural homomorphism $f_1 \oright f_2 \colon \mathrm{G}_1 \oright \mathrm{G}_2 \rightarrow \mathrm{H}$, sending $(u_1,u_2) \in \mathrm{G}_1 \oright \mathrm{G}_2$ to $f_1(u_1) + f_2(u_2)$, which has the following properties:
\begin{enumerate}
\item[(a)] $f_1 \oright f_2$ is spherical 
\item[(b)] ${T}_{f_1 \oright f_2} = {T}_{f_1} \cdot {T}_{f_2}$. 
\item[(c)] If the homomorphisms $f_1$ and $f_2$ are relative $(-1)^n$-CY and $\mathrm{H}$ is a $(-1)^{n+1}$-CY pseudolattice, for some $n \in \bZ$, then $f$ is $(-1)^n$-CY.
\end{enumerate}
\end{proposition}
\begin{proof}
For simplicity of notation, throughout this proof we let $f := f_1 \oright f_2$. We begin by proving (a) and (b). We claim first that the right adjoint to $f$ is the map $r\colon \mathrm{H} \rightarrow \mathrm{G}_1 \oright \mathrm{G}_2$ defined by
\[r(v) := \left(r_1T_{f_2}(v), r_2(v)\right) = \left(r_1(v) - r_1f_2r_2(v), r_2(v)\right),\]
where $r_1$ and $r_2$ are right adjoints to $f_1$ and $f_2$, respectively. To see this note that, by linearity,
\[
\langle f(u_1,u_2),v \rangle_\mathrm{H} = \langle f_1(u_1), v \rangle_\mathrm{H} + \langle f_2(u_2) , v \rangle_\mathrm{H}.
\]
On the other hand, we have that
\begin{align*}
\langle (u_1,u_2), r(v) \rangle_{\mathrm{G}_1 \oright \mathrm{G}_2} &= \langle u_1, r_1(v) \rangle_{\mathrm{G}_1} - \langle u_1, r_1f_2r_2(v) \rangle_{\mathrm{G}_1} + \\ &\quad + \langle f_1(u_1), f_2 r_2(v) \rangle_\mathrm{H} + \langle u_2, r_2(v) \rangle_{\mathrm{G}_2}.
\end{align*}
Using that fact that $r_1$ is a right adjoint to $f_1$, the second term on the right hand side is equal to $\langle f_1(u_1), f_2r_2(v) \rangle_{\mathrm{H}}$, so the second and third terms cancel and we are left with
\[\langle (u_1,u_2), r(v) \rangle_{\mathrm{G}_1 \oright \mathrm{G}_2} = \langle u_1, r_1(v) \rangle_{\mathrm{G}_1} + \langle u_2 , r_2(v) \rangle_{\mathrm{G}_2} = \langle f_1(u_1), v \rangle_\mathrm{H} + \langle f_2(u_2), v \rangle_\mathrm{H}.
\]
Therefore $r$ is right adjoint to $f$. 

Next we check that
\begin{align*}
T_f &= (\mathrm{id}_\mathrm{H} - fr)(v)\\ 
&= v - f_1(r_1(v) - r_1f_2r_2(v)) - f_2r_2(v)\\
&= (\mathrm{id}_\mathrm{H} - f_1 r_1)(\mathrm{id}_\mathrm{H} - f_2 r_2)(v) \\
&= {T}_{f_1} \cdot {T}_{_{f_2}}.
\end{align*}
This proves (b) and, since both ${T}_{f_1}$ and ${T}_{f_2}$ are invertible, so is ${T}_f$. Thus $f$ is a spherical homomorphism, proving (a).

Now let us prove (c). Assume that ${C}_{f_i} = (-1)^n S_{\mathrm{G}_i}$ for each $i \in \{1,2\}$ and $S_\mathrm{H} = (-1)^{n+1}\mathrm{id}_\mathrm{H}$. We have
\begin{align*}
{C}_f(u_1,u_2) &= (\mathrm{id}_{\mathrm{G}_1 \oright \mathrm{G}_2} - rf)(u_1,u_2) \\
&= (u_1 - r_1(f_1(u_1) + f_2(u_2)) + r_1f_2r_2(f_1(u_1) + f_2(u_2)), \\
&\quad\quad u_2 - r_2(f_1(u_1) + f_2(u_2)))  \\ 
&=({C}_{f_1}(u_1) - r_1f_2{C}_{f_2}(u_2) + r_1f_2r_2f_1(u_1),\, {C}_{f_2}(u_2) - r_2f_1(u_1)).
\end{align*} 
Using this we compute
\begin{align*}
\langle (w_1,w_2), {C}_f(u_1,u_2) \rangle_{\mathrm{G}_1 \oright \mathrm{G}_2} &= \langle w_1, {C}_{f_1}(u_1) \rangle_{\mathrm{G}_1} - \langle w_1, r_1f_2{C}_{f_2}(u_2) \rangle_{\mathrm{G}_1}+ \\ 
&\quad + \langle w_1, r_1f_2r_2f_1(u_1) \rangle_{\mathrm{G}_1}  + \langle f_1(w_1), f_2{C}_{f_2}(u_2) \rangle_\mathrm{H} - \\
&\quad - \langle f_1(w_1), f_2r_2f_1(u_1) \rangle_\mathrm{H}  + \langle w_2, {C}_{f_2}(u_2) \rangle_{\mathrm{G}_2}  - \\ 
&\quad - \langle w_2, r_2f_1(u_1) \rangle_{\mathrm{G}_2}.
\end{align*}
Using the right adjoint property of $r_1$ and $r_2$ on the second, third, and final terms on the right-hand side of this expression, we may reduce it to
\[
\langle w_1,{C}_{f_1}(u_1) \rangle_{\mathrm{G}_1} + \langle w_2, {C}_{f_2}(u_2) \rangle_{\mathrm{G}_2} - \langle f_2(w_2), f_1(u_1) \rangle_{\mathrm{H}}.
\]
Now, using the fact that ${C}_{f_i} = (-1)^{n}S_{\mathrm{G}_i}$ and  $ S_\mathrm{H} = (-1)^{n+1}\mathrm{id}_\mathrm{H}$, we obtain
\begin{align*}
\langle (w_1, w_2) ,\mathrm{C}_{f}(u_1, u_2) \rangle_{\mathrm{G}_1 \oright \mathrm{G}_2 } &= (-1)^n(\langle u_1, w_1 \rangle_{\mathrm{G}_1} + \langle u_2, w_2 \rangle_{\mathrm{G}_2} + \langle f_1(u_1), f_2(w_2) \rangle_{\mathrm{H}}) \\
&= (-1)^n\langle (u_1,u_2) ,(w_1,w_2) \rangle_{\mathrm{G}_1 \oright \mathrm{G}_2}
\end{align*}
for any $(u_1,u_2)$ and $(w_1,w_2)$ in $\mathrm{G}_1 \oright \mathrm{G}_2$. It follows by uniqueness of the operator $S_{\mathrm{G}_1 \oright \mathrm{G}_2}$ that ${C}_f = (-1)^nS_{\mathrm{G}_1 \oright \mathrm{G}_2}$. This proves (c).
\end{proof}

For us, the following example provides one of the main applications of these ideas.

\begin{example}\label{ex:directed}
Suppose we have a pair of elements $v_1,v_2$ in a $(-1)^1$-CY pseudolattice $\mathrm{H}$ (for instance, $\mathrm{H}$ could be the elliptic curve pseudolattice $\mathrm{E}$). Then $\langle v_i,v_i \rangle_{\mathrm{H}} =  0$ for each $i$, as the bilinear form on $\mathrm{H}$ is antisymmetric. Define rank one $(-1)^0$-CY pseudolattices $\mathrm{Z}(v_1)$ and $\mathrm{Z}(v_2)$ generated by elements $z_1$ and $z_2$, respectively, with $\langle z_i,z_i \rangle_{\mathrm{Z}(v_i)} = 1$. By Example \ref{ex:sphob}, we have relative $(-1)^0$-CY spherical homomorphisms $\zeta_i \colon \mathrm{Z}(v_i) \to \mathrm{H}$ taking $z_i$ to $v_i$.

Let $\mathrm{Z}(v_1,v_2)$ denote the unimodular pseudolattice $\mathrm{Z}(v_1,v_2) := \mathrm{Z}(v_1) \oright_{\mathrm{H}} \mathrm{Z}(v_2)$. The underlying abelian group of $\mathrm{Z}(v_1,v_2)$ is freely generated by $z_1$ and $z_2$, and its bilinear form is given by the Gram matrix
\[
\begin{pmatrix} 1 & \langle v_1, v_2 \rangle_{\mathrm{H}} \\ 0 & 1  \end{pmatrix};
\]
$(z_1,z_2)$ thus form an exceptional basis for $Z(v_1,v_2)$.

By Proposition \ref{prop:glueCY}, we have that the homomorphism $\zeta\colon \mathrm{Z}(v_1,v_2) \to \mathrm{H}$ sending $z_i$ to $v_i$ is spherical and relative $(-1)^0$-CY (so the cotwist $C_{\zeta} = S_{\mathrm{Z}(v_1,v_2)}$). The twist $T_f$ is given by the product $T_{\zeta_1} \cdot T_{\zeta_2}$, where
\[T_{\zeta_i}(w) = w - \langle v_i, w \rangle_{\mathrm{G}} v_i\]
as in Example \ref{ex:sphob}.

More generally, given an ordered collection $(v_1,\ldots,v_{n})$ of objects in a $(-1)^{1}$-CY pseudolattice $\mathrm{H}$, we may apply this construction repeatedly to obtain a unimodular pseudolattice $\mathrm{Z}(v_1,\ldots,v_{n})$ whose underlying group has basis $(z_1,\dots, z_{n})$ and whose bilinear form is given by
\[
\langle z_i, z_j \rangle_{\mathrm{Z}(v_1,\ldots,v_{n})} = \begin{cases}1 & \text{ if } i = j \\
\langle v_i, v_j \rangle_\mathrm{H} & \text{ if } i < j \\
0 & \text{ if } i > j.
\end{cases}
\]
The homomorphism $\zeta\colon \mathrm{Z}(v_1,\ldots,v_{n}) \rightarrow \mathrm{G}$ sending $z_i$ to $v_i$ is spherical and $(-1)^{0}$-CY (so ${C}_{\zeta} = S_{\mathrm{Z}(v_1,\ldots,v_{n})}$), and ${T}_{\zeta} = {T}_{\zeta_{1}}  \cdots {T}_{\zeta_{n}}$. Moreover, $(z_1,\ldots,z_n)$ form an exceptional basis for $\mathrm{Z}(v_1,\ldots,v_{n})$.
\end{example}

\section{Quasi del Pezzo Pseudolattices}\label{sec:qdp}

Next we focus our attention on a special case of the theory above, where we have a spherical homomorphism whose target is the elliptic curve pseudolattice $\mathrm{E}$ from Definition \ref{def:E}. 

\subsection{Spherical homomorphisms to the elliptic curve} \label{sec:qdpprops}

We begin by establishing a number of useful properties of pseudolattices admitting spherical homomorphisms to the elliptic curve pseudolattice $\mathrm{E}$. 

\begin{proposition}\label{prop:issurfacelike}
Suppose $(a,b)$ is a basis for $\mathrm{E}$ as in Definition \ref{def:E}. Let $f\colon \mathrm{G} \to \mathrm{E}$ be a relative $(-1)^0$-CY spherical homomorphism with right adjoint $r$. Then $\mathrm{G}$ is surface-like with point-like vector $\bp = r(a)$ if and only if $r(a) \in \mathrm{G}$ is primitive and the twist $\mathrm{T}_f$ satisfies $T_f(a) = a$.
\end{proposition}

\begin{proof} We begin by showing that  if $r(a) \in \mathrm{G}$ is primitive and the twist $\mathrm{T}_f$ satisfies $T_f(a) = a$, then $\mathrm{G}$ is surface-like with point-like vector $\bp = r(a)$. To do this, we need to show that $r(a)$ satisfies conditions (1)--(3) from Definition \ref{def:surfacelike}. To check (1) note that, according to our hypotheses, we have $a = \mathrm{T}_f(a) = a - fr(a)$. Thus $fr(a) = 0$, so $\langle r(a) ,r(a) \rangle_\mathrm{G} = \langle  fr(a), a \rangle_\mathrm{E} = 0$.

To check (2), for any $u \in \mathrm{G}$ we know that $\langle r(a), u \rangle_\mathrm{G} = \langle u, S_\mathrm{G}r(a) \rangle_\mathrm{G}$. But, as $f$ is relative $(-1)^0$-CY, $S_\mathrm{G} = C_f = \mathrm{id}_\mathrm{G} - rf$, so
\[
\langle r(a), u \rangle_\mathrm{G} = \langle u, r(a) \rangle_\mathrm{G} - \langle u, rfr(a) \rangle_\mathrm{G}.
\]
Since $fr(a) = 0$, it follows that $\langle u, rfr(a) \rangle_\mathrm{G} =0$, and thus $\langle r(a), u \rangle_\mathrm{G } = \langle u ,r(a) \rangle_\mathrm{G}$.

Finally, we check condition (3). Assume that $u_1,u_2 \in \mathrm{G}$ satisfy $\langle u_i , r(a) \rangle_{\mathrm{G}} = 0$. This condition is equivalent to $\langle f(u_i), a \rangle_{\mathrm{E}} =0$, by adjunction. Therefore, there are some constants $c_1,c_2 \in \bZ$ so that $f(u_i) = c_i a$, as $\prescript{\perp}{}a = \Span(a)$ in $\mathrm{E}$. So, by Lemma \ref{lem:useful},
\[\langle u_1,u_2\rangle_\mathrm{G} - \langle u_2, u_1 \rangle_\mathrm{G} = \langle f(u_1), f(u_2) \rangle_\mathrm{E} = \langle c_{1} a, c_{2}a \rangle_\mathrm{E} = 0.\]
Therefore, $\langle u_1, u_2 \rangle_\mathrm{G} = \langle u_2, u_1 \rangle_\mathrm{G}$, as required.

For the converse statement, primitivity of $r(a)$ is assured by Definition \ref{def:surfacelike}. To prove the twist condition, note that it is equivalent to show that $fr(a) = 0$. Since $r(a)$ is point-like, we obtain that $\langle fr(a), a \rangle_{\mathrm{E}} = \langle r(a), r(a) \rangle_{\mathrm{G}} = 0$, so we must have $fr(a) = ca$ for some $c \in \bZ$, as $\prescript{\perp}{}a = \Span(a)$ in $\mathrm{E}$. 

It therefore suffices to show that $c=0$. Since $f$ is relative $(-1)^0$-CY, $S_\mathrm{G} = C_f = \mathrm{id}_\mathrm{G} - rf$, so for any $u \in \mathrm{G}$ we have
\[
\langle r(a), u \rangle_\mathrm{G} = \langle u, r(a) \rangle_\mathrm{G} - \langle u, rfr(a) \rangle_\mathrm{G} = \langle u, r(a) \rangle_\mathrm{G} - c\langle u, r(a) \rangle_\mathrm{G}.
\]
Moreover, as $r(a)$ is point-like we have $\langle r(a), u \rangle_\mathrm{G } = \langle u ,r(a) \rangle_\mathrm{G}$, so it follows that $c\langle u, r(a) \rangle_\mathrm{G} = 0$ for all $u \in \mathrm{G}$. But $r(a)$ is primitive and the bilinear form on $\mathrm{G}$ is non-degenerate, so this is only possible if $c = 0$. Thus $fr(a) = 0$ and hence $T_f(a) = a$, as required.\end{proof}

Note that we have two distinguished classes $a$ and $b$ in $\mathrm{E}$, but so far we have only made use of $a$. We claim that $[r(b)]$ plays the role of the anticanonical class in the N\'{e}ron-Severi group $\mathrm{NS}(\mathrm{G})$.

\begin{proposition}\label{prop:anticanonical}
Suppose $(a,b)$ is a basis for $\mathrm{E}$ as in Definition \ref{def:E}. Let $f\colon \mathrm{G} \to \mathrm{E}$ be a relative $(-1)^0$-CY spherical homomorphism with right adjoint $r$ and suppose that  $\mathrm{G}$ is surface-like with point-like vector $\bp = r(a)$. Then we have $r(b) \in \mathbf{p}^\perp$ and $[r(b)] = -K_\mathrm{G}$.
\end{proposition}
\begin{proof}
The first claim is easy to prove: note that
\[
\langle\mathbf{p}, r(b) \rangle_{\mathrm{G}} = \langle r(a),r(b) \rangle_\mathrm{G} = \langle fr(a), b \rangle_\mathrm{E}
\]
and, since $T_f(a)=a$ by Proposition \ref{prop:issurfacelike}, we have $fr(a) = 0$, so $\langle \mathbf{p}, r(b) \rangle_\mathrm{G} = 0$.

To prove the second statement it is necessary to show that, for any $u_1,u_2 \in\mathrm{G}$,
\[
\langle u_1 , u_2 \rangle_\mathrm{G} - \langle u_2, u_1 \rangle_\mathrm{G} = q(r(b), \lambda(u_1 \wedge u_2) ).
\]
Noting that $q$ is symmetric, we begin by computing that
\begin{align*}
q(\lambda(u_1 \wedge u_2), r(b)) &=  -\langle \langle u_1 ,r(a) \rangle_\mathrm{G} u_2 - \langle u_2, r(a) \rangle_\mathrm{G} u_1, r(b) \rangle_\mathrm{G}\\
& =  -\langle u_1, r(a) \rangle_\mathrm{G} \langle u_2, r(b) \rangle_\mathrm{G} + \langle u_2, r(a) \rangle_\mathrm{G} \langle u_1, r(b) \rangle_\mathrm{G} \\ 
& = - \langle f(u_1), a \rangle_\mathrm{E} \langle f(u_2), b \rangle_\mathrm{E} + \langle f(u_2), a \rangle_\mathrm{E} \langle f(u_1), b \rangle_\mathrm{E}
\end{align*}
Note that there are pairs of integers $(p_1,q_1)$ and $(p_2,q_2)$ so that
\[
f(u_1) = p_1 a + q_1 b, \quad\quad f(u_2) = p_2 a + q_2 b.
\]
Then $p_i = -\langle f(u_i), b \rangle_\mathrm{E}$, and $q_i = \langle f(u_i), a \rangle_\mathrm{E}$ for $i =1,2$. Therefore, 
\begin{align*}
-\langle f(u_1), a \rangle_\mathrm{E} \langle f(u_2), b \rangle_\mathrm{E} + \langle f(u_2), a \rangle_\mathrm{E} \langle f(u_1), b \rangle_\mathrm{E}& = q_1p_2 - q_2 p_1 \\
&= \langle f(u_1) , f(u_2) \rangle_\mathrm{E}
\end{align*}
and, finally, Lemma \ref{lem:useful} gives $\langle f(u_1), f(u_2) \rangle_\mathrm{E} = \langle u_1, u_2 \rangle_\mathrm{G} - \langle u_2,u_1 \rangle_{\mathrm{G}}$. This completes the proof.
\end{proof}

Our final result shows that we can explicitly compute the twist $T_f$ in this setting.

\begin{proposition}\label{prop:matrix}
Suppose $(a,b)$ is a basis for $\mathrm{E}$ as in Definition \ref{def:E}. Let $f\colon \mathrm{G} \to \mathrm{E}$ be a relative $(-1)^0$-CY spherical homomorphism with right adjoint $r$ and suppose that $\mathrm{G}$ is surface-like with point-like vector $\bp = r(a)$. Then in the basis $(a,b)$ the twist $T_f$ has matrix $\begin{psmallmatrix} 1 & -d \\ 0 & 1\end{psmallmatrix}$, where $d := q(K_{\mathrm{G}},K_{\mathrm{G}})$. 
\end{proposition}
\begin{proof} Begin by noting that, by Proposition \ref{prop:issurfacelike}, $T_f(a) = a$. To find the matrix of $T_f$ it remains to compute $T_f(b) = b - fr(b)$.  By Proposition \ref{prop:anticanonical} we have $r(b) \in \mathbf{p}^{\perp}$ and $\mathbf{p}^{\perp} = \prescript{\perp}{}\bp$ as $\mathrm{G}$ is surface-like, so  
\[0 = \langle r(b), \mathbf{p} \rangle_{\mathrm{G}} =  \langle r(b), r(a) \rangle_{\mathrm{G}} = \langle fr(b), a \rangle_{\mathrm{E}}.\]
Thus $fr(b) \in \prescript{\perp}{}a = \Span(a)$, so $fr(b) = da$ for some $d \in \bZ$. This proves that $T_f(b) = b - da$, so the matrix of $T_f$ has the required form.

It remains to show that $d =q(K_{\mathrm{G}},K_{\mathrm{G}})$. By adjunction, 
\[
q(K_\mathrm{G}, K_\mathrm{G}) = -\langle r(b), r(b) \rangle_\mathrm{G} = -\langle fr(b) , b \rangle_\mathrm{E}.
\]
But $fr(b) = da$ by the argument above, so $\langle fr(b), b \rangle_\mathrm{E} = \langle da,b \rangle_\mathrm{E} = -d$. This shows that $q(K_{\mathrm{G}},K_{\mathrm{G}}) = d$, as required.
\end{proof}

\begin{remark} \label{rem:defect} Under the assumptions of Proposition \ref{prop:matrix}, it is immediate from the definition that the defect $\delta(\mathrm{G}) = 0$ if and only if $\rank(\mathrm{G}) = 12 - d$.
\end{remark}

\subsection{Quasi del Pezzo homomorphisms}

In this subsection we will give the central definition of this paper. However, in order to state it, we first define a useful and well-known invariant of a lattice, which unfortunately suffers from some ambiguity in the literature. Let $\mathrm{L}$ be a nondegenerate lattice (in the traditional sense of a finitely-generated free abelian group equipped with an integral symmetric bilinear form). Then the Gram matrix of the real extension $\mathrm{L} \otimes \mathbb{R}$ can always be written as a diagonal matrix, for some choice of basis, with all entries either $+1$ or $-1$. 

\begin{definition} \label{def:signature}
The pair of numbers $(n_+, n_-)$ counting the number of $+1$ and $-1$ entries, respectively, in the diagonal representation of the Gram matrix is called the \emph{signature} of $\mathrm{L}$. Define $\sigma(\mathrm{L}) :=n_+ - n_-$.
\end{definition}

\begin{remark}
The invariant $\sigma(\mathrm{L})$ is also often called the signature of $\mathrm{L}$; we will avoid this nomenclature for the sake of clarity.
\end{remark}

We now come to the central definition of this paper.

\begin{definition}\label{def:qdp}
A spherical homomorphism of pseudolattices $f\colon \mathrm{G} \to \mathrm{E}$ is called a \emph{quasi del Pezzo homomorphism} if there exists a basis $(a,b)$ for $\mathrm{E}$, as in Definition \ref{def:E}, so that all of the following conditions hold. 
\begin{enumerate}[(1)]
\item $\mathrm{G}$ is surface-like with point-like vector $\mathbf{p} = r(a)$.
\item $f$ is relative $(-1)^0$-Calabi-Yau.
\item $\mathrm{G}$ admits an exceptional basis $e_{\bullet} = (e_1,\ldots,e_n)$ such that $f(e_i) \in \mathrm{E}$ is primitive for each $1 \leq i \leq n$.
\item $\NS(\mathrm{G})$ has signature $(1,\mathrm{rank}(\NS(\mathrm{G})) - 1)$.
\end{enumerate}
\end{definition}

\begin{remark}\label{rem:equivalentcondition} By Proposition \ref{prop:issurfacelike}, condition (1) in this definition is equivalent to
\begin{enumerate}[(1')]
\item $r(a) \in \mathrm{G}$ is primitive and the twist $T_f$ satisfies $T_f(a) = a$.
\end{enumerate}
Note also that condition (3) implies that $\mathrm{G}$ is unimodular. Finally, condition (4) should be thought of as a weaker version of Kuznetsov's \emph{geometric} condition \cite[Definition 4.3]{ecslc}; by \cite[Lemma 4.5]{ecslc}, all pseudolattices arising in the usual way as $\mathrm{K}_0^{\mathrm{num}}(\mathbf{D}(X))$ for a smooth projective surface $X$ are geometric.
\end{remark}

\begin{definition} We say that two quasi del Pezzo homomorphisms $f_1\colon \mathrm{G}_1 \to \mathrm{E}$ and $f_2\colon \mathrm{G}_2 \to \mathrm{E}$ are \emph{isomorphic} if
\begin{itemize}
\item there are bases $(a_1,b_1)$ and $(a_2,b_2)$ for $\mathrm{E}$, so that $(a_i,b_i)$ satisfies the conditions of Definition \ref{def:qdp} for $f_i\colon \mathrm{G}_i \to \mathrm{E}$, and an automorphism $\varphi \colon \mathrm{E} \to \mathrm{E}$ taking $(a_1,b_1)$ to $(a_2,b_2)$,
\item there is an isomorphism of pseudolattices $\psi\colon \mathrm{G}_1 \to \mathrm{G}_2$ (i.e. an isomorphism of the underlying abelian groups which preserves the bilinear form), and
\item we have $\varphi f_1 = f_2 \psi$.
\end{itemize}
It is easy to check that the last part of this definition implies that $\psi r_1 = r_2 \varphi$, where $r_i$ is the right adjoint to $f_i$.
\end{definition}

The motivation behind Definition \ref{def:qdp} is the following example.

\begin{example} \label{ex:qdp} Let $X$ be a nonsingular rational surface and suppose that the anticanonical linear system $|-K_X|$ contains a smooth member $C$; such surfaces are called \emph{quasi del Pezzo surfaces}. By Example \ref{ex:relativeCY}, we have a relative $(-1)^0$-CY spherical homomorphism $i^*\colon \mathrm{K}_0^{\mathrm{num}}(\mathbf{D}(X)) \to \mathrm{K}_0^{\mathrm{num}}(\mathbf{D}(C))$, given by the derived pull-back under the inclusion $i\colon C \hookrightarrow X$. This homomorphism has right adjoint $i_*$, given by derived push-forward. We will show that  $i^*\colon \mathrm{K}_0^{\mathrm{num}}(\mathbf{D}(X)) \to \mathrm{K}_0^{\mathrm{num}}(\mathbf{D}(C))$ is a quasi del Pezzo homomorphism.

Identify $\mathrm{K}_0^{\mathrm{num}}(\mathbf{D}(C)) \cong \mathrm{E}$ as in Example \ref{ex:elliptic}(1): identify $a$ with the class $\{\calO_p\}$ of the structure sheaf $\calO_p$ of a point $p \in C$ and $b$ with the class $\{\calO_C\}$ of $\calO_C$. Then $i_*(a) = i_*(\{\calO_p\})$ is the class of the structure sheaf of a point in $X$, which is point-like in $\mathrm{K}_0^{\mathrm{num}}(\mathbf{D}(X))$ by Example \ref{ex:dx}. We have thus shown that conditions (1) and (2) hold.

To prove (3) and (4), note that it is not difficult to show (c.f. \cite{sdp}, \cite[Proposition 0.4]{asdps}) that if $(X,C)$ satisfies the assumptions above, then $(X,C)$ is either 
\begin{itemize}
\item a blow-up of a smooth cubic curve in $\bP^2$ in $k \geq 0$ points in almost general position (a set of points is in \emph{almost general position} if no stage of the blowing-up involves blowing up a point which lies on a rational $(-2)$-curve; infinitely near points are allowed as long as this condition is not violated), 
\item a smooth curve of bidegree $(2,2)$ in $\bP^1 \times \bP^1$, or 
\item a smooth anticanonical curve in the Hirzebruch surface $\bF_2$.
\end{itemize}

It is well-known that a full exceptional collection for $\mathbf{D}(\bP^2)$ is given by the line bundles $(\calO,\calO(1),\calO(2))$. Moreover, using work of Orlov \cite{pbmtdccs}, one can compute full exceptional collections for $\mathbf{D}(\bP^1 \times \bP^1)$ and $\mathbf{D}(\bF_2)$ (see, for example, Auroux, Katzarkov, and Orlov \cite[Section 2.7]{mswppnd} or Elagin and Lunts \cite[Section 2.5]{feclbdps}); for example, the collections $(\calO,\calO(1,0),\calO(0,1),\calO(1,1))$ (for $\bP^1\times \bP^1$) and $(\calO,\calO(f),\calO(s),\calO(s+f))$ (for $\bF_2$; here $s$ is the class of the $(-2)$-section and $f$ the class of a fibre) are full and exceptional. By Example \ref{ex:Grothendieck}, these exceptional collections give rise to exceptional bases in $\mathrm{K}_0^{\mathrm{num}}(\mathbf{D}(X))$ for $X = \bP^2$, $\bP^1 \times \bP^1$, and $\bF_2$, and it is easy to check that these exceptional bases satisfy condition (3).

It is also well-understood how blow-ups affect $\mathbf{D}(X)$ \cite[Section 4]{pbmtdccs} \cite[Section 2.1]{msdpsvccs}. Indeed, if $X$ is a surface such that $\mathrm{K}_0^{\mathrm{num}}(\mathbf{D}(X))$ has exceptional basis $\calB$, and $\varphi\colon \widetilde{X} \to X$ is a blow-up of a point in $X$ with exceptional curve $E$, then an exceptional basis for $\mathrm{K}_0^{\mathrm{num}}(\mathbf{D}(\widetilde{X}))$ is given by $(\varphi^*\calB, \{\calO_{E}\})$. Using this, it is easy to check inductively that all quasi del Pezzo surfaces admit exceptional bases satisfying condition (3); intuitively, this condition simply says that all exceptional $(-1)$-curves meet $C$ in a single point, which is an immediate consequence of the genus formula.

Finally, it is well known that $\NS(X)$ is an odd unimodular lattice of signature $(1,k)$ if $X$ is a blow-up of $\bP^2$ in $k$ points and an even unimodular lattice of signature $(1,1)$ if $X \cong \bP^1 \times \bP^1$ or $\bF_2$. Thus condition (4) holds, and $i^*\colon \mathrm{K}_0^{\mathrm{num}}(\mathbf{D}(X)) \to \mathrm{K}_0^{\mathrm{num}}(\mathbf{D}(C))$ is a quasi del Pezzo homomorphism.

In this setting, the results of Subsection \ref{sec:qdpprops} reproduce well-known facts.
\begin{itemize}
\item Proposition \ref{prop:issurfacelike} states that $T_f(a) = a$, i.e. that $i_*i^*(\{\calO_p\}) = 0$.
\item Proposition \ref{prop:anticanonical} states that $i_*(\{\calO_C\}) = \{\calO_X\} - \{\calO_X(-C)\}$ is anticanonical in $\NS(X)$. But, in $\NS(X)$, we have $[\calO_X] \equiv 0$ and $ - [\calO_X(-C)] \equiv [\calO_X(C)]$, so we obtain the familiar result that $[\calO_X(C)]$ is an anticanonical class. 
\item Proposition \ref{prop:matrix} states that
\[T_{i^*}(\{\calO_C\}) = \{\calO_C\} - i^*i_*(\{\calO_C\}) = \{\calO_C\} - d\{\calO_p\},\] 
where $d$ is the self-intersection $K_X^2 = C^2$; this is easily verified. 
\item As we know that the defect $\delta(\mathrm{K}_0^{\mathrm{num}}(\mathbf{D}(X))) = 0$ (see Example \ref{ex:dx}), Remark \ref{rem:defect} allows us to compute that $\rank(\mathrm{K}_0^{\mathrm{num}}(\mathbf{D}(X))) = 12 - d$.
\end{itemize}
\end{example}

\subsection{Classification of quasi del Pezzo homomorphisms}

In this section we will prove the first main result of this paper, which will show that the quasi del Pezzo homomorphisms given in Example \ref{ex:qdp} are, essentially, all that exist. To perform this classification we will run a modified version of Kuznetsov's \cite{ecslc} minimal model program for exceptional bases of pseudolattices. In order to describe this program, we first need a notion of minimality for exceptional bases.
 
\begin{definition} \cite[Section 4]{ecslc}
An exceptional basis $e_\bullet = (e_1,\ldots,e_n)$ of a surface-like pseudolattice is called \emph{norm minimal} if the sum
\[
\sum_{i=1}^d \rank(e_i)^2
\]
is minimal among all exceptional bases mutation equivalent to $e_\bullet$.
\end{definition}

Clearly, every exceptional basis can be mutated to obtain a norm minimal exceptional basis. Kuznetsov proves a number of results about norm-minimal exceptional bases; the important one for us is as follows.

\begin{theorem}\label{thm:kuz} \textup{\cite[Corollary 4.22]{ecslc}}
Let $\mathrm{G}$ be a surface-like pseudolattice so that the signature of $\mathrm{NS}(\mathrm{G})$ is $(1, \mathrm{rank}(\mathrm{NS}(\mathrm{G})) -1)$. If $e_\bullet$ is a norm minimal exceptional basis of $\mathrm{G}$ consisting of elements of non-zero rank, then the rank of $\mathrm{G}$ is $3$ or $4$.
\end{theorem}

\begin{remark}
In \cite{ecslc}, Theorem \ref{thm:kuz} is stated slightly differently. The condition that  $\mathrm{NS}(\mathrm{G})$ has signature $(1,\mathrm{rank}(\NS(\mathrm{G})) - 1)$ is replaced with the condition that $\mathrm{G}$ be \emph{geometric}, which imposes the additional condition that $K_\mathrm{G}$ satisfies
\begin{equation}\label{eq:characteristic}
q(K_\mathrm{G}, u) \equiv q(u,u) \bmod 2
\end{equation}
for each $u \in \mathrm{NS}(\mathrm{G})$. However, the proof of \cite[Corollary 4.22]{ecslc} only uses the signature condition, not the condition that (\ref{eq:characteristic}) holds.
\end{remark}

\begin{remark} Note that $\rank(\mathrm{G}) \geq 3$ for any $\mathrm{G}$ satisfying the conditions of Theorem \ref{thm:kuz}, as $1 \leq \rank(\NS(\mathrm{G})) = \rank(\mathrm{G}) - 2$ by the signature assumption.
\end{remark}

Let $f\colon \mathrm{G} \to \mathrm{E}$ be a quasi del Pezzo homomorphism and let $e_{\bullet} = (e_1,\ldots,e_n)$ be an exceptional basis as in Definition \ref{def:qdp}. We wish to mutate $e_{\bullet}$ to get a norm-minimal exceptional basis. However, in order to do this we must first show that condition (3) of Definition \ref{def:qdp} is preserved under mutations. This follows from the next lemma.

\begin{lemma} \label{lem:mutation} Let $f\colon \mathrm{G} \to \mathrm{E}$ be a quasi del Pezzo homomorphism and let $e_{\bullet} = (e_1,\ldots,e_n)$ be an exceptional basis as in Definition \ref{def:qdp}. Then for any $1 \leq i \leq n-1$, the elements $f(\bL_{e_i}(e_{i+1}))$ and $f(\bR_{e_{i+1}}(e_i))$ are primitive in $\mathrm{E}$.
\end{lemma}
\begin{proof}  By Lemma \ref{lem:useful}, for any $1 \leq i \leq n-1$ we have
\begin{equation} \label{eq:fcommutes} \langle f(e_i),f(e_{i+1}) \rangle_{\mathrm{E}} =  \langle e_i,e_{i+1} \rangle_{\mathrm{G}} - \langle e_{i+1}, e_i \rangle_{\mathrm{G}} = \langle e_i,e_{i+1} \rangle_{\mathrm{G}},\end{equation}
where we have used the fact that $e_{\bullet}$ is exceptional.

 Now write $f(e_i) = p_i a + q_i b$ in a basis $(a,b)$ of $\mathrm{E}$ such that Definition \ref{def:qdp} holds. By Equation \eqref{eq:fcommutes} we have $f(\bL_{e_i}(e_{i+1})) = f(e_{i+1}) - \langle f(e_i),f(e_{i+1}) \rangle_{\mathrm{E}} f(e_i)$, from which it is easy to compute that
\[ f(\bL_{e_i}(e_{i+1})) = \begin{pmatrix} 1 - p_iq_i & p_i^2 \\ -q_i^2 & 1+p_iq_i \end{pmatrix} f(e_{i+1}),\]
where we treat elements of $\mathrm{E}$ as column vectors in the basis $(a,b)$. But this matrix is an element of $\mathrm{SL}(2,\bZ)$, so its action takes primitive elements to primitive elements. This proves the statement for left mutations; the statement for right mutations follows by an analogous computation (or from the fact that they are the inverses of left mutations).\end{proof}

Therefore, after mutation, we may assume that the exceptional basis $e_{\bullet}$ of Definition \ref{def:qdp}(3) is norm-minimal. Then the hypotheses of Theorem \ref{thm:kuz} are satisfied, so we see that if $\rank(\mathrm{G}) > 4$, then $e_{\bullet}$ must contain elements of zero rank. Such elements are characterized by the following lemma.

\begin{lemma}\label{lem:mmp}
Let $f\colon \mathrm{G} \to \mathrm{E}$ be a quasi del Pezzo homomorphism. Choose a basis $(a,b)$ for $\mathrm{E}$ and an exceptional basis $e_{\bullet} = (e_1,\ldots,e_n)$ for $\mathrm{G}$ so that Definition \ref{def:qdp} holds. Then for any $1 \leq i \leq n$, $\rank(e_i) = 0$ if and only if $f(e_i) = \pm a \in \mathrm{E}$.
\end{lemma}
\begin{proof}
By Proposition \ref{prop:issurfacelike}, we have the point-like vector $\mathbf{p} = r(a) \in \mathrm{G}$. By definition, $\rank(e_i) = 0$ if and only if $\langle e_i, r(a) \rangle_{\mathrm{G}} = 0$ and, by adjunction, this is equivalent to $\langle f(e_i), a \rangle_\mathrm{E} = 0$. Moreover, we know that $\prescript{\perp}{}a = \mathrm{Span}(a)$, so $\rank(e_i) = 0$ if and only if $f(e_i)$ is a multiple of $a$. But $f(e_i)$ is primitive, so $f(e_i) = \pm a$.
\end{proof}

Using this, we obtain the following theorem, which should be thought of as a more refined version of Kuznetsov's notion of contraction \cite[Section 5.1]{ecslc}.

\begin{theorem} \label{thm:contraction} Let $f\colon \mathrm{G} \to \mathrm{E}$ be a quasi del Pezzo homomorphism with $\rank(\mathrm{G}) > 4$. Choose a basis $(a,b)$ for $\mathrm{E}$ and an exceptional basis $e_{\bullet} = (e_1,\ldots,e_n)$ for $\mathrm{G}$ so that Definition \ref{def:qdp} holds. Let $\zeta\colon \mathrm{Z}(a) \to \mathrm{E}$ be defined as in Example \ref{ex:sphob} and let $z \in \mathrm{Z}(a)$ denote its generator.

Then there is a quasi del Pezzo homomorphism $f'\colon \mathrm{G}' \to \mathrm{E}$ with $\rank(\mathrm{G}') = \rank(\mathrm{G})-1$ and an exceptional basis $e_{\bullet}'$ for $\mathrm{G}'$, satisfying Definition \ref{def:qdp}, so that, after performing a series of mutations on $e_{\bullet}$, we have an isomorphism of pseudolattices $\psi$ that fits into the following diagram
\[\xymatrix{\mathrm{G} \ar[rr]^-{\psi} \ar[rd]^{f} & & \mathrm{G}' \oright_{\mathrm{E}}  \mathrm{Z}(a)\ar[ld]_{f' \oright \zeta} \\
& \mathrm{E} &
}\]
satisfying $\psi(e_{\bullet}) = (e_{\bullet}',z)$. 
\end{theorem}
\begin{proof} As argued above, after a series of mutations we may assume that $\mathrm{G}$ has an exceptional basis $e_{\bullet} = (e_1,\ldots,e_n)$ which is norm-minimal and satisfies Definition \ref{def:qdp}.  Theorem \ref{thm:kuz} then gives that $e_{\bullet}$ contains an element $e_{\alpha}$ of rank zero and so, after possibly multiplying $e_{\alpha}$ by $(-1)$, we may assume that $f(e_{\alpha}) = a \in \mathrm{E}$, by Lemma \ref{lem:mmp}. After a series of left mutations, we may convert $e_{\bullet}$ to a new exceptional basis $\hat{e}_{\bullet} := (e_1',\ldots,e_{n-1}',e_{\alpha})$. By Lemma \ref{lem:mutation}, $\hat{e}_{\bullet}$ satisfies Definition \ref{def:qdp}.

Let $\mathrm{G}'$ denote the sub-pseudolattice of $\mathrm{G}$ defined by $\mathrm{G}' := e_{\alpha}^{\perp}$ and let $f'\colon \mathrm{G}' \to \mathrm{E}$ be the induced homomorphism of abelian groups; note that $e_{\bullet}' := (e_1',\ldots,e_{n-1}')$ forms an exceptional basis for $\mathrm{G}'$ and $\rank(\mathrm{G}') = \rank(\mathrm{G})-1$. We claim that $f'\colon \mathrm{G}' \to \mathrm{E}$ satisfies the conclusion of the theorem. The proof proceeds in three steps.

\textbf{Step 1.} We first show that $f'$ is a spherical homomorphism. Define a map $s\colon \mathrm{E} \to \mathrm{G}'$ by
\[s(v) := r(v)  - \langle e_{\alpha},r(v) \rangle_{\mathrm{G}} e_{\alpha},\]
where $r$ is the right adjoint to $f$. We begin by checking that $s(v) \in \mathrm{G}'$ for all $v \in \mathrm{E}$, so that $s$ is well-defined. Indeed, write $r(v) =   \sum_{i=1}^{n-1} k_i e_i' + k_{\alpha}e_{\alpha}$ in terms of the exceptional basis $\hat{e}_{\bullet}$ of $\mathrm{G}$, for $k_i, k_{\alpha}\in \bZ$. Then $\langle e_{\alpha}, r(v) \rangle_{\mathrm{G}} = k_{\alpha}$, as $\hat{e}_{\bullet}$ is exceptional, so $s(v) = \sum_{i=1}^{n-1} k_i e_i' \in \mathrm{G}'$.

Next define $r'\colon \mathrm{E} \to \mathrm{G}'$ by
\[r'(v) = s(v) + \langle e_{\alpha},r(v)\rangle_{\mathrm{G}} s(a) = r(v)  + \langle e_{\alpha},r(v) \rangle_{\mathrm{G}}(r(a) - e_{\alpha}),\] 
where we have used that fact that $\langle e_{\alpha},r(a) \rangle_{\mathrm{G}} = \langle f(e_{\alpha}),a \rangle_{\mathrm{G}} = \langle a,a \rangle_{\mathrm{E}} = 0$ to compute $s(a)$.

We claim that $r'$ is a right adjoint to $f'$. Indeed, for any $u \in \mathrm{G}'$ and $v \in \mathrm{E}$, we have
\[\langle u,r'(v) \rangle_{\mathrm{G}'} = \langle u, r(v) \rangle_{\mathrm{G}} + \langle e_{\alpha},r(v) \rangle_{\mathrm{G}}\left(  \langle u, r(a) \rangle_{\mathrm{G}} -  \langle u, e_{\alpha}\rangle_{\mathrm{G}} \right) \]
Now note that, using Lemma \ref{lem:useful}, for any $u \in \mathrm{G}'$ we have
\begin{equation} \label{eq:raealpha}\langle u, r(a) \rangle_{\mathrm{G}} = \langle f(u), a \rangle_{\mathrm{E}} = \langle f(u), f(e_{\alpha}) \rangle_{\mathrm{E}} = \langle u, e_{\alpha} \rangle_{\mathrm{G}} - \langle e_{\alpha}, u\rangle_{\mathrm{G}} = \langle u, e_{\alpha} \rangle_{\mathrm{G}},\end{equation}
as $\langle e_{\alpha}, u\rangle_{\mathrm{G}} = 0$. So we have
\[\langle u,r'(v) \rangle_{\mathrm{G}'} =  \langle u, r(v) \rangle_{\mathrm{G}} = \langle f(u), v \rangle_{\mathrm{E}} = \langle f'(u), v \rangle_{\mathrm{E}},\]
and $r'$ is a right adjoint to $f'$.

For $u \in \mathrm{G}'$, the cotwist endomorphism is given by 
\begin{align*}C_{f'}(u) &= u - r'f'(u) \\
&= u - rf(u)  - \langle e_{\alpha},rf(u) \rangle_{\mathrm{G}}(r(a) - e_{\alpha}) \\
&= C_f(u)  - \langle e_{\alpha},rf(u) \rangle_{\mathrm{G}}(r(a) - e_{\alpha}).\end{align*}
Then for any $u_1,u_2 \in \mathrm{G}'$, we have
\begin{align*}\langle u_1,u_2 \rangle_{\mathrm{G}'} &=  \langle u_1,u_2 \rangle_{\mathrm{G}} \\
&= \langle u_2,C_f(u_1) \rangle_{\mathrm{G}}\\
&= \langle u_2,C_{f'}(u_1) \rangle_{\mathrm{G}} + \langle e_{\alpha},rf(u) \rangle_{\mathrm{G}}\left(\langle u_2,r(a) \rangle_{\mathrm{G}} - \langle u_2,e_{\alpha} \rangle_{\mathrm{G}}\right)\\
&= \langle u_2,C_{f'}(u_1) \rangle_{\mathrm{G}'},\end{align*}
where we have used Equation \eqref{eq:raealpha} and the fact that $S_{\mathrm{G}} = C_{f}$ (as $f$ is $(-1)^0$-CY). Thus we see that $C_{f'}$ is the Serre operator $S_{\mathrm{G}'}$ for $\mathrm{G}'$, which is invertible as $\mathrm{G}'$ is unimodular (since it has an exceptional basis). We therefore conclude that $f'\colon \mathrm{G}' \to \mathrm{E}$ is spherical.

\textbf{Step 2.} We next show that $f'\colon \mathrm{G}' \to \mathrm{E}$ is a quasi del Pezzo homomorphism. To do this we need to check the conditions of Definition \ref{def:qdp}. Condition (3) is satisfied by the exceptional basis $e_{\bullet}'$ for $\mathrm{G}'$, by construction, and Condition (2) is just the fact (proved in Step 1) that $C_{f'} = S_{\mathrm{G}'}$.

We prove the equivalent condition (1') in place of condition (1) (see Remark \ref{rem:equivalentcondition}). Note that, since $\langle e_{\alpha}, r(a) \rangle_{\mathrm{G}} = \langle f(e_{\alpha}), a \rangle_{\mathrm{E}} = \langle a,a \rangle_{\mathrm{E}} = 0$, we have $r'(a) = r(a)$. Thus, since $r(a)$ is primitive, so is $r'(a)$, and we compute that
\[T_{f'}(a) = a - f'r'(a) = a - fr(a) = T_f(a) = a,\]
as required.

Finally, for condition (4), note that the point-like vector $\mathbf{p} = r(a)$ lies in $\mathrm{G}'$, so $\NS(\mathrm{G}')$ is equal to the orthogonal complement of $e_{\alpha}$ in $\NS(\mathrm{G})$. But $q(e_{\alpha},e_{\alpha}) = -1$ in $\NS(\mathrm{G})$, as $e_{\alpha}$ is exceptional, so $\NS(\mathrm{G}')$ has signature $(1, \rank(\NS(\mathrm{G})) - 2)$, as required.

\textbf{Step 3.} Finally, we prove that there is an isomorphism of pseudolattices $\psi\colon \mathrm{G} \to \mathrm{G}' \oright_{\mathrm{E}}  \mathrm{Z}(a)$ that commutes with the projections to $\mathrm{E}$. Let $z$ denote the generator of $\mathrm{Z}(a)$; recall that $\langle z,z \rangle_{\mathrm{Z}(a)} =1$ and the spherical homomorphism $\zeta\colon\mathrm{Z}(a) \to \mathrm{E}$ maps $z$ to $a$, by definition.

Now let $u \in \mathrm{G}$ be any element. Write $u$ in terms of the exceptional basis $\hat{e}_{\bullet}$ for $\mathrm{G}$ as $u = \sum_{i=1}^{n-1} k_i e_i' + k_{\alpha}e_{\alpha} = u' + k_{\alpha}e_{\alpha}$, where $u' := \sum_{i=1}^{n-1} k_i e_i' \in \mathrm{G}'$ and $k_i, k_{\alpha}\in \bZ$. Then define $\psi(u) := \left(u',\ k_{\alpha}z \right) \in  \mathrm{G}' \oright_{\mathrm{E}}  \mathrm{Z}(a)$. The map $\psi$ is clearly an isomorphism on the underlying abelian groups and we have
\[(f' \oright \zeta)\psi(u) = f'\left(\sum_{i=1}^{n-1} k_i e_i'\right) + k_{\alpha} a = \sum_{i=1}^{n-1} k_i f(e_i') + k_{\alpha} f(e_{\alpha}) = f(u).\]
Finally, let $u_1,u_2 \in \mathrm{G}$ be any two elements and write $u_j = u_j' + k_{\alpha j}e_{\alpha}$ with $u_j' \in \mathrm{G}'$ and $k_{\alpha j} \in \bZ$, for each $j \in \{1,2\}$. Then
\begin{align*}\langle u_1,u_2 \rangle_{\mathrm{G}} &= \langle u_1', u_2'\rangle_{\mathrm{G}} + \langle u_1',  k_{\alpha 2}e_{\alpha}\rangle_{\mathrm{G}} + \langle k_{\alpha 1}e_{\alpha} , k_{\alpha 2}e_{\alpha} \rangle_{\mathrm{G}} \\
&= \langle u_1', u_2'\rangle_{\mathrm{G}'}+ \langle f(u_1'),  k_{\alpha 2}a \rangle_{\mathrm{E}} + \langle k_{\alpha 1}z , k_{\alpha 2}z \rangle_{\mathrm{Z}(a)} \\
&= \langle \psi(u_1), \psi(u_2) \rangle_{\mathrm{G}' \oright_{\mathrm{E}}  \mathrm{Z}(a)},\end{align*}
where the second equality follows from Equation \eqref{eq:raealpha}. Hence $\psi$ is an isomorphism of pseudolattices. This completes the proof.\end{proof}

\begin{remark} \label{rem:contraction} In the setting of Example \ref{ex:qdp}, proceeding from $\mathrm{G}$ to $\mathrm{G}'$ should be thought of as ``contracting an exceptional curve''. Since performing such a contraction on a quasi del Pezzo surface produces another quasi del Pezzo surface, Theorem \ref{thm:contraction} should not be too surprising.
\end{remark}

By repeated application of Theorem \ref{thm:contraction}, we reduce the problem of classifying quasi del Pezzo homomorphisms $f\colon \mathrm{G} \to \mathrm{E}$ to the cases  where $\rank(\mathrm{G}) = 3$ or $4$. These are dealt with in the next two propositions.

%\begin{corollary} \label{cor:>4classification}
%Let $f\colon \mathrm{G} \to \mathrm{E}$ be a quasi del Pezzo homomorphism and let $(a,b)$ be a basis for $\mathrm{E}$ so that definition \ref{def:qdp} holds.. Then there is a quasi del Pezzo homomorphism $f'\colon \mathrm{G}' \to \mathrm{E}$ with $\rank(\mathrm{G}') \leq 4$ and a commutative diagram
%\[\xymatrix{\mathrm{G} \ar[rr]^-{\psi} \ar[rd]^{f} & & \mathrm{G}' \oright_{\mathrm{E}}  \mathrm{Z}(a,\ldots,a)\ar[ld]_{g} \\
%& \mathrm{E} &
%}\]
%such that $\psi$ is an isomorphism of pseudolattices, $\mathrm{Z}(a,\ldots,a)$ is as in Example \ref{ex:directed}, and $g$ is the spherical homomorphism induced from $f'$ and $\mathrm{Z}(a,\ldots,a) \to \mathrm{E}$ by Proposition \ref{prop:glueCY}. 
%\end{corollary}

\begin{proposition} \label{prop:3classification} Let $f\colon \mathrm{G} \to \mathrm{E}$ be a quasi del Pezzo homomorphism with $\rank(\mathrm{G}) = 3$. Choose a basis $(a,b)$ for $\mathrm{E}$ and an exceptional basis $e_{\bullet}$ for $\mathrm{G}$ so that Definition \ref{def:qdp} holds. Then after performing a series of mutations on $e_{\bullet}$, the Gram matrix for $\langle \cdot , \cdot \rangle_{\mathrm{G}}$ becomes
\[\chi := \begin{pmatrix} 1 & 3 & 6 \\ 0 & 1 & 3 \\ 0 & 0 & 1 \end{pmatrix}.\]
Moreover, after a change of basis in $\mathrm{E}$ of the form $\begin{psmallmatrix} 1 & k \\ 0 & 1 \end{psmallmatrix}$, for some $k \in \bZ$, the homomorphisms $f$ and $r$ are given by the matrices
\[F := \begin{pmatrix} 0 & 3 & 6 \\ 1 & 1 & 1 \end{pmatrix} \quad \text{and} \quad R := \begin{pmatrix} 1 & -9 \\ -2 & 15 \\ 1 & -6 \end{pmatrix}.\]
\end{proposition}

\begin{proof}
We begin by noting that, since $\mathrm{G}$ is a surface-like pseudolattice, a result of Kuznetsov \cite[Corollary 3.15]{ecslc} shows that the Serre automorphism $S_{\mathrm{G}}$ is unipotent. We may thus apply a result of Bondal and Polishchuk \cite[Example 3.2]{hpaamh} (see also \cite[Lemma 3.1.2]{ameecl4}) to see that if $e_{\bullet}$ is an exceptional basis of $\mathrm{G}$ (which we may assume satisfies Definition \ref{def:qdp}), then the Gram matrix for $\langle \cdot , \cdot \rangle_{\mathrm{G}}$ is
\[\begin{pmatrix} 1 & x & y \\ 0 & 1 & z \\ 0 & 0 & 1 \end{pmatrix},\]
where $x,y,z \in \bZ$ solve the classical Markov equation
\[x^2 + y^2 + z^2 - xyz = 0.\]
Markov \cite{fqbi} showed that, given any solution of this equation, all other solutions may be obtained by a process of mutation, which corresponds precisely to mutation of exceptional bases. It thus suffices to find any solution to the Markov equation; we choose $(x,y,z) = (3,6,3)$.

To compute the homomorphisms $f$ and $r$ we note first that, using the characterization of point-like vectors in \cite[Lemma 3.3]{ecslc}, the point-like vector $\bp \in \mathrm{G}$ may be written (after possibly multiplying our exceptional basis by $(-1)$) as $(1, -2, 1)$. Since $\bp = r(a)$ by definition, this gives the first column of $R$, and the right adjoint condition gives the second row of $F$. Now, using the fact that the Serre operator $S_{\mathrm{G}}$ has Gram matrix $\chi^{-1} \chi^T$, and the relative $(-1)^0$-CY condition, which gives $S_{\mathrm{G}} = \mathrm{id}_{\mathrm{G}} - rf$, we obtain that
\[F = \begin{pmatrix} k & k+3 & k+6 \\ 1 & 1 & 1 \end{pmatrix}, \quad R = \begin{pmatrix} 1 & -k-9 \\ -2 & 2k + 15 \\ 1 & -k - 6\end{pmatrix},\]
for some $k \in \bZ$, and it is easy to check that the corresponding maps $f$ and $r$ satisfy Definition \ref{def:qdp}. To conclude the proof, it suffices to note that the matrices above may be converted into the required form by conjugating by $\begin{psmallmatrix} 1 & k \\ 0 & 1 \end{psmallmatrix}$. 
\end{proof}

We thus see that, up to isomorphism, there is only one quasi del Pezzo homomorphism with $\rank(\mathrm{G}) = 3$. We henceforth denote it by $f_3\colon \mathrm{G}_3 \to \mathrm{E}$. A straightforward computation shows that $f_3\colon \mathrm{G}_3 \to \mathrm{E}$, with its exceptional basis $(e_1,e_2,e_3)$ as in Proposition \ref{prop:3classification}, may be identified with $\zeta\colon \mathrm{Z}(b,3a+b,6a+b) \to \mathrm{E}$ with its standard exceptional basis $(z_1,z_2,z_3)$, as in Example \ref{ex:directed}.

\begin{remark} \label{rem:3class} It is easy to see that $\mathrm{G}_3$ with its exceptional basis $(e_1,e_2,e_3)$ is isomorphic to $\mathrm{K}_0^{\mathrm{num}}(\mathbf{D}(\bP^2))$ with the standard exceptional basis $(\calO, \calO(1), \calO(2))$, as one should expect from Example \ref{ex:qdp}. In this basis the spherical homomorphism $i^*$ and its adjoint $i_*$ have matrices $F$ and $R$ respectively. 
\end{remark}

\begin{proposition} \label{prop:4classification} Let $f\colon \mathrm{G} \to \mathrm{E}$ be a quasi del Pezzo homomorphism with $\rank(\mathrm{G}) = 4$. Choose a basis $(a,b)$ for $\mathrm{E}$ and an exceptional basis $e_{\bullet}$ for $\mathrm{G}$ so that Definition \ref{def:qdp} holds. Then either $\mathrm{G} \cong \mathrm{G}_3 \oright_{\mathrm{E}} \mathrm{Z}(a)$, as in Theorem \ref{thm:contraction}, or, after performing a series of mutations on $e_{\bullet}$, the Gram matrix for $\langle \cdot , \cdot \rangle_{\mathrm{G}}$ becomes
\[\chi := \begin{pmatrix} 1 & 2 & 2 & 4 \\ 0 & 1 & 0 & 2 \\ 0 & 0 & 1 & 2 \\ 0 & 0 & 0 & 1 \end{pmatrix},\]
and these two cases are distinct. Moreover, in the second case, after a change of basis in $\mathrm{E}$ of the form $\begin{psmallmatrix} 1 & k \\ 0 & 1 \end{psmallmatrix}$, for some $k \in \bZ$, the homomorphisms $f$ and $r$ are given by the matrices
\[F := \begin{pmatrix} 0 & 2 & 2 & 4 \\ 1 & 1 & 1 & 1 \end{pmatrix} \quad \text{and} \quad R := \begin{pmatrix} 1 & -8 \\ -1 & 6 \\ -1 & 6 \\ 1 & -4 \end{pmatrix}.\]
\end{proposition}

\begin{proof}As in the proof of Proposition \ref{prop:3classification} we begin by noting that, since $\mathrm{G}$ is a surface-like pseudolattice, a result of Kuznetsov \cite[Corollary 3.15]{ecslc} shows that the Serre automorphism is unipotent. Thus we may apply a result of de Thanhoffer de Volcsey and Van den Bergh \cite[Theorem A]{ameecl4} to see that, after a series of mutations, we may convert any exceptional basis for $\mathrm{G}$ (which we may assume satisfies Definition \ref{def:qdp}) into one in which the Gram matrix for $\langle \cdot , \cdot \rangle_{\mathrm{G}}$ is either
\[\begin{pmatrix} 1 & 2 & 2 & 4 \\ 0 & 1 & 0 & 2 \\ 0 & 0 & 1 & 2 \\ 0 & 0 & 0 & 1 \end{pmatrix}\quad \text{or} \quad \begin{pmatrix} 1 & n & 2n & n \\ 0 & 1 & 3 & 3 \\ 0 & 0 & 1 & 3 \\ 0 & 0 & 0 & 1 \end{pmatrix},\]
for some $n \in \bN$, and these cases are distinct. 

In the first case, we use the same method as in the proof of Proposition \ref{prop:3classification} to compute $F$ and $R$. The point-like vector is given by $\bp = (1,-1,-1,1)$ and, omitting the details, we obtain
\[F := \begin{pmatrix} k & k+2 & k+2 & k+4 \\ 1 & 1 & 1 & 1 \end{pmatrix}, \quad R := \begin{pmatrix} 1 & -k-8 \\ -1 & k+6 \\ -1 & k+6 \\ 1 & -k-4 \end{pmatrix},\]
for some $k \in \bZ$. It is easy to check that the corresponding maps $f$ and $r$ satisfy Definition \ref{def:qdp} and that the matrices above may be converted into the required form by conjugating by $\begin{psmallmatrix} 1 & k \\ 0 & 1 \end{psmallmatrix}$. 

Now suppose that we are in the second case. Choose an exceptional basis $e_{\bullet} := (e_1,e_2,e_3,e_4)$, satisfying Definition \ref{def:qdp}, for which the Gram matrix for $\langle \cdot , \cdot \rangle_{\mathrm{G}}$ is of the second type, above. Define $c := f(e_2) \in \mathrm{E}$; note that $c$ is primitive by assumption. Write $c = pa + qb$ for some $p,q \in \bZ$. As $c$ is primitive, we must have $\mathrm{gcd}(p,q) = 1$, so there exist $p',q' \in \bZ$ with $p'p+q'q = 1$. Define $d = -q'a + p'b \in \mathrm{E}$. Then we have 
\[\langle c,d \rangle_\mathrm{E} = -1, \quad\quad \langle d,c \rangle_\mathrm{E} = 1, \quad\quad \langle c,c \rangle_\mathrm{E} = \langle d,d \rangle_\mathrm{E} = 0,\]
and $(c,d)$ is an alternative basis for $\mathrm{E}$.

By Lemma \ref{lem:useful}, we see that whenever $i < j$ we have $\langle f(e_i),f(e_j) \rangle_{\mathrm{E}} =  \langle e_i, e_j \rangle_{\mathrm{G}}$. Therefore $\langle f(e_2),f(e_3) \rangle_{\mathrm{E}} = \langle f(e_2),f(e_4) \rangle_{\mathrm{E}} = \langle f(e_3),f(e_4) \rangle_{\mathrm{E}} = 3$, so we may compute that $f(e_3) = (x+1)c - 3d$ and $f(e_4) = xc - 3d$, for some $x \in \bZ$. By primitivity of $f(e_3)$ and $f(e_4)$, we thus obtain that $x \equiv 1 \pmod 3$. Finally, by computing $\langle f(e_1), f(e_i)\rangle_{\mathrm{E}}$ for $i \in \{2,3,4\}$, we obtain
\[f(e_1) = \tfrac{1}{3}(1-x)n c + n d.\]
This is primitive if and only if $n =1$, so we find that the Gram matrix in this basis is
\[\begin{pmatrix} 1 & 1 & 2 & 1 \\ 0 & 1 & 3 & 3 \\ 0 & 0 & 1 & 3 \\ 0 & 0 & 0 & 1 \end{pmatrix}.\]

To complete the proof note that, using the characterization of point-like vectors in \cite[Lemma 3.3]{ecslc}, the point-like vector $\bp \in \mathrm{G}$ is given in the basis $e_{\bullet}$ by $(0,1,-1,1)$. We thus have $\rank(e_1) = 0$. Following the proof of Theorem \ref{thm:contraction}, we therefore see that $\mathrm{G} \cong \mathrm{G}_3 \oright_{\mathrm{E}} \mathrm{Z}(a)$ in this case.
\end{proof}

We thus find that, up to isomorphism, there are two quasi del Pezzo homomorphisms with $\rank(\mathrm{G}) = 4$. One of these is given by $f_3\oright\zeta\colon \mathrm{G}_3 \oright_{\mathrm{E}} \mathrm{Z}(a) \to \mathrm{E}$; we henceforth denote the other by $f_4\colon \mathrm{G}_4 \to \mathrm{E}$. They are easily distinguished by observing that $f_3\oright\zeta$ is surjective, but $f_4$ has cokernel $\bZ/2\bZ$.

As in the rank 3 case, a straightforward computation shows that $f_4\colon \mathrm{G}_4 \to \mathrm{E}$, with its exceptional basis $(e_1,e_2,e_3,e_4)$ as in Proposition \ref{prop:4classification}, may be identified with $\zeta\colon \mathrm{Z}(b,2a+b,2a+b,4a+b) \to \mathrm{E}$ with its standard exceptional basis $(z_1,z_2,z_3,z_4)$, as in Example \ref{ex:directed}.

\begin{remark} \label{rem:4class} It is  easy to see that $\mathrm{G}_4$ with its exceptional basis $(e_1,e_2,e_3,e_4)$ is isomorphic to  $\mathrm{K}_0^{\mathrm{num}}(\mathbf{D}(\bP^1 \times \bP^1))$ with exceptional basis $(\calO, \calO(1,0),\calO(0,1), \calO(1,1))$, as one might expect from Example \ref{ex:qdp}. In this basis the spherical homomorphism $i^*$ and its adjoint $i_*$ have matrices $F$ and $R$ respectively.
\end{remark}

Putting these results together, we see that we have at most two quasi del Pezzo homomorphisms for each $\rank(\mathrm{G}) > 3$. To complete the classification, we should check whether these ever coincide. To do this, we use the identifications of $\mathrm{G}_3$ and $\mathrm{G}_4$ with the pseudolattices $\mathrm{Z}(v_1,\ldots,v_n)$. 

\begin{proposition} Choose a basis $(a,b)$ for $\mathrm{E}$. Let $\zeta \colon \mathrm{Z}(b,3a+b,6a+b,a,a) \to \mathrm{E}$ and $\zeta'\colon\mathrm{Z}(b,2a+b,2a+b,4a+b,a) \to \mathrm{E}$ be defined as in Example \ref{ex:directed} and let $z_{\bullet}$, $z_{\bullet}'$ denote the standard exceptional bases for $\mathrm{Z}(b,3a+b,6a+b,a,a)$ and $\mathrm{Z}(b,2a+b,2a+b,4a+b,a)$ respectively. Then, after performing a series of mutations on the exceptional basis $z_{\bullet}$, we have an isomorphism of pseudolattices $\psi$ that fits into the following diagram
\[\xymatrix{\mathrm{Z}(b,3a+b,6a+b,a,a) \ar[rr]^-{\psi} \ar[rd]^{\zeta} & & \mathrm{Z}(b,2a+b,2a+b,4a+b,a)\ar[ld]_{\zeta'} \\
& \mathrm{E} &
}\]
such that $\psi(z_{\bullet}) = z_{\bullet}'$.
\end{proposition}
\begin{proof} Writing the maps $\zeta \colon \mathrm{Z}(b,3a+b,6a+b,a,a) \to \mathrm{E}$ and $\zeta'\colon\mathrm{Z}(b,2a+b,2a+b,4a+b,a) \to \mathrm{E}$ explicitly in terms as matrices, as was done in Propositions \ref{prop:3classification} and \ref{prop:4classification}, converts this into a linear algebra problem: it suffices to find a sequence of mutations taking the matrix representation of $\zeta \colon \mathrm{Z}(b,3a+b,6a+b,a,a) \to \mathrm{E}$ in the exceptional basis $z_{\bullet}$ to the matrix representation of $\zeta'\colon\mathrm{Z}(b,2a+b,2a+b,4a+b,a) \to \mathrm{E}$ in the exceptional basis $z'_{\bullet}$. Such a sequence of mutations is given explicitly by 
\[z'_{\bullet} = \bR_{4,5} \circ \bL_{3,4} \circ \bL_{2,3} \circ \bR_{4,5} \circ \bR_{3,4}(z_{\bullet}).\]
\end{proof}

\begin{remark} One may rephrase this proposition as an isomorphism between $\mathrm{G}_3 \oright_{\mathrm{E}} \mathrm{Z}(a) \oright_{\mathrm{E}} \mathrm{Z}(a) $ and $\mathrm{G}_4 \oright_{\mathrm{E}} \mathrm{Z}(a)$. In light of Remarks \ref{rem:contraction}, \ref{rem:3class} and \ref{rem:4class}, this should be thought of as an analogue of the isomorphism between the blow-up of $\bP^2$ in two points and the blow-up of $\bP^1 \times \bP^1$ in a single point, on the level of numerical Grothendieck groups.
\end{remark}

Consequently, we find that there is precisely one quasi del Pezzo homomorphism, up to isomorphism, with $\rank(\mathrm{G}) = 5$. Putting everything together, we obtain the following theorem, which is the main result of this paper. Roughly speaking, it states that the classification of quasi del Pezzo homomorphisms parallels that of quasi del Pezzo surfaces (see Example \ref{ex:qdp}).

\begin{theorem} \label{thm:main} Let $f\colon \mathrm{G} \to \mathrm{E}$ be a quasi del Pezzo homomorphism and let $n:=\rank(\mathrm{G})$. Choose a basis $(a,b)$ for $\mathrm{E}$ and an exceptional basis $e_{\bullet}$ for $\mathrm{G}$ so that Definition \ref{def:qdp} holds. Then after performing a series of exceptional mutations on $e_{\bullet}$ and an automorphism of $\mathrm{E}$ of the form $\begin{psmallmatrix} 1 & k \\ 0 & 1 \end{psmallmatrix}$ \textup{(}written in the basis $(a,b)$\textup{)}, for $k \in \bZ$, we have that $f\colon \mathrm{G} \to \mathrm{E}$ with its exceptional basis $e_{\bullet}$ is either:
\begin{itemize}
\item $\mathrm{Z}(b,3a+b,6a+b,a,\ldots,a) \to \mathrm{E}$ with its standard exceptional basis as in Example \ref{ex:directed}, for any $n \geq 3$; or
\item $\mathrm{Z}(b,2a+b,2a+b,4a+b) \to \mathrm{E}$ with its standard exceptional basis as in Example \ref{ex:directed}, for $n = 4$.
\end{itemize} 
\end{theorem}

\begin{corollary}Let $f\colon \mathrm{G} \to \mathrm{E}$ be a quasi del Pezzo homomorphism and let $n:=\rank(\mathrm{G})$. If $(a,b)$ is a basis for $\mathrm{E}$ so that  Definition \ref{def:qdp} holds, then the twist endomorphism $T_f$ is given in the basis $(a,b)$ by 
\[\begin{pmatrix} 1 & n-12 \\ 0 & 1\end{pmatrix}.\]
Consequently, the defect $\delta(\mathrm{G}) = 0$.
\end{corollary} 
\begin{proof} This is a fairly simple consequence of Theorem \ref{thm:main}. It is easy to compute explicitly that the twist endomorphism has the required form for the special quasi del Pezzo homomorphisms $f_3\colon \mathrm{G}_3 \to \mathrm{E}$ and $f_4\colon \mathrm{G}_4 \to \mathrm{E}$. Moreover, by Proposition \ref{prop:glueCY}, the twist associated to $f_3\oright \zeta \colon \mathrm{G}_3 \oright_{\mathrm{E}} \mathrm{Z}(a,\ldots,a) \to \mathrm{E}$ is the product $T_{f_3} \cdot T_{\zeta}^{n-3}$,  where $T_{\zeta}$ is the twist associated to $\mathrm{Z}(a) \to \mathrm{E}$. But this latter twist is just the twist $\begin{psmallmatrix}1 & 1 \\ 0 & 1\end{psmallmatrix}$, by Example \ref{ex:sphob}. To complete the argument, note that twists of this form are not affected by conjugation by matrices of the form $\begin{psmallmatrix} 1 & k \\ 0 & 1 \end{psmallmatrix}$, so they are unchanged by the automorphisms of $\mathrm{E}$ from Theorem \ref{thm:main}. Finally, the statement about the defect is an immediate consequence of Proposition \ref{prop:matrix} and Remark \ref{rem:defect}.
\end{proof}

\section{Genus $1$ Fibrations over Discs}\label{sec:fibrations}

The aim of this section is to apply the results above to the problem of classifying a certain class of genus $1$ fibrations over discs.

We begin with some general setup that will be used throughout this section. Let $\Delta := \{x \in \bC : |x| \leq 1\}$ denote the closed complex unit disc and let $\pi\colon Y \to \Delta$ denote a surjective map from an oriented, smooth, compact $4$-manifold $Y$ to $\Delta$, whose smooth fibres are $2$-tori. Let $\Sigma := \{x_1,\ldots,x_n\} \subset \Delta$ denote the set of critical values of $\pi$, which we assume are all distinct. Assume further that $\Sigma$ is contained in the interior $\Delta^{\circ}$ of $\Delta$ and that all critical points are modelled on the complex Morse singularity $(z_1,z_2) \mapsto z_1^2 + z_2^2$ in an orientation-preserving coordinate chart.  We call such $\pi\colon Y \to \Delta$ (sometimes abbreviated simply to $Y$, when no confusion is likely to result) satisfying these assumptions a \emph{genus $1$ Lefschetz fibration}.

\subsection{Genus $1$ Lefschetz fibrations} \label{sec:ellipticlf} 

We begin with some generalities on genus $1$ Lefschetz fibrations  $\pi\colon Y \to \Delta$. These ideas are fairly well-known (see, for instance, \cite[Section 8]{4mkc}); our presentation broadly follows that given by Auroux \cite{fsseisf}. Choose a smooth fibre $C$ over a point $x_{\infty}$ on the boundary $\partial\Delta$ of $\Delta$, along with paths $\gamma_i \subset \Delta \setminus \Sigma$ from $x_{\infty}$ to $x_i$, assumed to be disjoint from each other and $\partial \Delta$ except for at the common starting point $x_{\infty}$. Relabelling as necessary, we may assume that the paths emanating from $x_{\infty}$ are labelled $\gamma_1,\gamma_2,\ldots,\gamma_n$ in a clockwise order.

The fibre $C$ over $x_{\infty}$ is a smooth $2$-torus; recall from Example \ref{ex:elliptic} that $\mathrm{H}_1(C,\bZ)$, equipped with the usual intersection form, is isomorphic to the pseudolattice $\mathrm{E}$.  
The singular fibre over a critical value $x_i$ is a nodal curve, obtained from $C$ by collapsing a simple closed curve $v_i \subset C$ called the \emph{vanishing cycle}; we denote the homology class of $v_i$ by  $[v_i] \in \mathrm{H}_1(C,\bZ)$.

Monodromy around a loop based at $x_{\infty}$ acts on $C$ as an element of the mapping class group $\mathrm{SL}(2,\bZ)$ of the $2$-torus.  More precisely, for each point $x_i$,  monodromy around the loop given by going out along the curve $\gamma_i$, around a small anticlockwise loop containing $x_i$, then back along $\gamma_i$ acts on $C$ as a Dehn twist $\tau_i \in \mathrm{SL}(2,\bZ)$ around the vanishing cycle $v_i$. Moreover, anticlockwise monodromy around $\partial\Delta$ gives a further element $\tau_{\infty} \in \mathrm{SL}(2,\bZ)$, which we call the \emph{total monodromy}. This setup gives a \emph{monodromy factorization} of $\tau_{\infty}$ as the product of the $\tau_i$, we write it as
\[\tau_{\infty} = (\tau_1,\tau_2,\ldots,\tau_n).\]

Choose generators $a,b \in \mathrm{E} \cong \mathrm{H}_1(C,\bZ)$ as in Definition \ref{def:E}, and let $\delta_a,\delta_b \in \mathrm{SL}(2,\bZ)$ denote Dehn twists along corresponding curves in $C$. The mapping class group is generated by $\delta_a$ and $\delta_b$ and these correspond, respectively, to the generators 
\[M_{1,0} := \begin{pmatrix} 1 & 1 \\ 0 & 1 \end{pmatrix} \quad \text{and} \quad M_{0,1} := \begin{pmatrix} 1& 0 \\ -1 & 1 \end{pmatrix}\]
of $\mathrm{SL}(2,\bZ)$. The Dehn twist about a simple closed curve in the class $pa + qb \in \mathrm{H}_1(C,\bZ)$ then corresponds to the matrix
\[M_{p,q} := \begin{pmatrix} 1-pq & p^2 \\ -q^2 & 1+pq \end{pmatrix} \in \mathrm{SL}(2,\bZ).\]

If $[v_i] := p_ia + q_ib \in \mathrm{H}_1(C,\bZ)$, then the Dehn twist $\tau_i$ associated to $\gamma_i$ corresponds to the matrix $M_{p_i,q_i}$. Thus the monodromy factorization
\[\tau_{\infty} = (\tau_1,\tau_2,\ldots,\tau_n)\]
may be rewritten as a matrix product
\[M_{\infty} = M_{p_1,q_1}M_{p_{2},q_{2}}\cdots M_{p_n,q_n},\]
where $M_{\infty}$ is the total monodromy matrix. 

It is well-known that, for a given choice of special fibre $C$ and paths $\gamma_i$, the monodromy factorization completely determines the structure of $Y$.  However, different choices of the paths $\gamma_i$ will lead to different monodromy factorizations for the same $Y$, as will different choices of basis $(a,b)$ for $\mathrm{H}_1(C,\bZ) \cong \mathrm{E}$. To account for this, note that the set of possible choices for the paths $\{\gamma_1,\ldots,\gamma_n\}$ (up to deformation in $\Delta\setminus \Sigma$) is acted on by the $n$-strand braid group $B_n$. The corresponding action of $B_n$ on monodromy factorizations is called \emph{Hurwitz equivalence} and is generated by the \emph{Hurwitz moves}
\[ (\tau_1,\ldots,\tau_{i-1},\tau_{i},\tau_{i+1},\tau_{i+2},\ldots,\tau_n) \sim (\tau_1, \ldots, \tau_{i-1},(\tau_{i}\tau_{i+1}\tau_{i}^{-1}),\tau_{i},\tau_{i+2},\ldots,\tau_n)\]
and their inverses; geometrically, the Hurwitz move above corresponds to deforming the path $\gamma_{i+1}$ across the point $x_{i}$. Moreover, a different choice of basis $(a,b)$ for $\mathrm{H}_1(C,\bZ) \cong \mathrm{E}$ acts on the monodromy factorization by \emph{global conjugation}, which simply replaces every $\tau_i$ with $\psi\tau_i\psi^{-1}$ for some $\psi \in \mathrm{SL}(2,\bZ)$ (note that this operation also changes $\tau_{\infty}$ to  $\psi\tau_{\infty}\psi^{-1}$).  

This justifies the following well-known theorem, which seems to be originally due to Moishezon \cite[Section II.2]{cscscpp}; this form appears in \cite[Section 2.1]{fsseisf}.

\begin{theorem} \label{thm:aurouxclass} The classification of genus $1$ Lefschetz fibrations $\pi\colon Y \to \Delta$, up to diffeomorphism, is equivalent to the classification of monodromy factorizations in $\mathrm{SL}(2,\bZ)$, up to Hurwitz equivalence and global conjugation. 
\end{theorem}

\subsection{Genus $1$ Lefschetz fibrations and pseudolattices}

We next develop the relationship between this fairly classical picture and the theory of pseudolattices developed in the previous sections. Most of the ideas here are not new, but there has been a resurgence in their study following the development of the theory of \emph{string junctions} in the physics literature. In our discussion we follow the recent mathematical treatment of string junctions by Grassi, Halverson, and Shaneson \cite{gtsj}. 

By parallel transporting the vanishing cycle $v_i$ along the path $\gamma_i$ we obtain a \emph{thimble} $t_i$. Let $J$ denote the free abelian group generated by the thimbles $\{t_1,\ldots,t_n\}$; its general element is a formal sum $\sum_{i=1}^n k_i t_i$, for $k_i \in \bZ$, such sums are called \emph{junctions}. Grassi, Halverson, and Shaneson \cite[Theorem 2.5]{gtsj} show that the natural map $J \to \mathrm{H}_2(Y,C;\bZ)$ taking a thimble $t_i$ to its class $[t_i]$ in relative homology is an isomorphism; we therefore have that $\mathrm{H}_2(Y,C;\bZ)$ is a free abelian group of rank $n$ generated by $\{[t_1],\ldots,[t_n]\}$. There is a natural map $\phi\colon \mathrm{H}_2(Y,C;\bZ) \to \mathrm{H}_1(C,\bZ)$ called the \emph{asymptotic charge}, arising from the long exact sequence of a pair, which takes the class $[t_i]$ of a thimble $t_i$  to the class $[v_i]$ of its corresponding vanishing cycle $v_i$. 

We can equip $\mathrm{H}_2(Y,C;\bZ)$ with a bilinear form as follows, which we call the \emph{Seifert pairing} after an analogous notion in singularity theory \cite[Section 2.3]{sdm2}. Let $\epsilon = 1 - \max_{z_i \in \Sigma}|z_i|>0$ and define a closed disc $\Delta' := \{z \in \bC: |z| \leq 1 - \epsilon\}$. Note that $\Sigma \subset \Delta' \subset \Delta$ by construction. Fix a group of diffeomorphisms $\varphi_{\theta} \colon \Delta \to \Delta$, for each $\theta \in \bR$, so that $\varphi_{\theta}$ fixes the disc $\Delta'$ and acts as an clockwise rotation through an angle of $\theta$ on $\partial\Delta$. Let $\psi_{\theta}$ denote a family of lifts of these diffeomorphisms to diffeomorphisms of $Y$.

\begin{definition} The \emph{Seifert pairing} 
\[\langle \cdot, \cdot \rangle_{\mathrm{Sft}} \colon \mathrm{H}_2(Y,C;\bZ) \times \mathrm{H}_2(Y,C;\bZ) \longrightarrow \bZ\]
is the map sending $(u_1,u_2)$ to the topological intersection product $\langle u_1,(\psi_{\pi})_*u_2 \rangle_{\mathrm{top}}$.
\end{definition}

Note that the definition of the Seifert pairing depends only upon $Y$; in particular, it does not depend upon the choices of the smooth fibre $C$ or the paths $\gamma_i$.

Given a basis $([t_1],\ldots,[t_n])$ of thimbles for $\mathrm{H}_2(Y,C;\bZ)$ (which does depend upon the choices of paths $\gamma_i$), we may compute the Seifert pairing explicitly. With notation as above, if $\phi([t_i]) = p_ia + q_ib$, the Seifert pairing on thimbles is given by
\[\langle [t_i],[t_j]\rangle_{\mathrm{Sft}} = \left\{\begin{array}{cl} q_ip_j - p_iq_j & \text{if } i < j \\ 1 & \text{if } i = j \\ 0 & \text{if } i > j \end{array}\right.\]
and may be extended to $\mathrm{H}_2(Y,C;\bZ)$ by linearity. 

\begin{remark} \label{rem:junctionpairing} The Seifert pairing defined above is not the same as the \emph{junction pairing} used in the string junction literature (see, for instance, \cite[Definition 2.8]{gtsj}). Instead, if $\chi$ is the Gram matrix of the Seifert pairing, the junction pairing arises as the symmetrized form $-\frac{1}{2}(\chi + \chi^T)$.
\end{remark}

The diffeomorphism $\psi_{-2\pi}$ induces an automorphism $(\psi_{-2\pi})_*\colon \mathrm{H}_2(Y,C;\bZ) \to \mathrm{H}_2(Y,C;\bZ)$. These concepts give us a pseudolattice.

\begin{proposition} \label{prop:plisZ} Let $\pi\colon Y \to \Delta$ be an genus $1$ Lefschetz fibration. 
\begin{enumerate}
\item $\mathrm{H}_2(Y,C;\bZ)$, equipped with the Seifert pairing, is a unimodular pseudolattice, with Serre operator given by $(\psi_{-2\pi})_*$
\item After fixing an ordered set of paths $\gamma_i$, the classes of the corresponding thimbles $([t_1],\ldots,[t_n])$ give an exceptional basis for $\mathrm{H}_2(Y,C;\bZ)$, such that $[v_i] := \phi([t_i]) \in \mathrm{H}_1(C,\bZ)$ is primitive for each $1 \leq i \leq n$. Hurwitz moves act on this exceptional basis as mutations.
\item Define $\zeta\colon \mathrm{Z}([v_1],\ldots,[v_n]) \to \mathrm{E}$ as in Example \ref{ex:directed}. Then there is an isomorphism $\mathrm{H}_2(Y,C;\bZ) \cong \mathrm{Z}([v_1],\ldots,[v_n])$ that commutes with the maps to $\mathrm{E}$, such that the twists $T_{\zeta_i}$ are the Dehn twists $\tau_i$ and the factorization $T_{\zeta} = T_{\zeta_1}\cdots T_{\zeta_n}$ is the monodromy factorization.
\end{enumerate}
\end{proposition}
\begin{proof} The only difficult part of (1) is the statement about the Serre operator. Indeed, note that $\langle u_2 , (\psi_{-2\pi})_* u_1\rangle_{\mathrm{Sft}}$ is equal to the topological intersection product $\langle u_2,(\psi_{-\pi})_*u_1 \rangle_{\mathrm{top}}$, and this latter intersection is equal to $\langle (\psi_{\pi})_*u_2,u_1 \rangle_{\mathrm{top}}$ as $\psi_{\pi}$ is a diffeomorphism. Finally, the topological intersection product is symmetric on $4$-manifolds, so  $\langle u_2 , (\psi_{-2\pi})_* u_1\rangle_{\mathrm{Sft}} = \langle u_1, (\psi_{\pi})_*u_2 \rangle_{\mathrm{top}} = \langle u_1, u_2\rangle_{\mathrm{Sft}}$, as required.

The proof of (2) is straightforward: the fact that $([t_1],\ldots,[t_n])$ forms an exceptional basis is immediate from the definition, and the classes $[v_i] \in \mathrm{H}_1(C,\bZ)$ are primitive as the vanishing cycles are all simple closed curves. The statement about mutations is a simple computation using the explicit description of the monodromy matrices $M_{p,q}$ and the Seifert pairing.

Finally, to prove (3), note that the isomorphism $\mathrm{H}_2(Y,C;\bZ) \to \mathrm{Z}([v_1],\ldots,[v_n])$ is just the map taking the exceptional basis of thimbles $([t_1],\ldots,[t_n])$ to the standard exceptional basis $(z_1,\ldots,z_n)$ of $\mathrm{Z}([v_1],\ldots,[v_n])$ (c.f. Example \ref{ex:directed}); it is easy to check that this map is an isomorphism of pseudolattices.
\end{proof}

The following corollary is immediate from Proposition \ref{prop:plisZ}(3) and the properties of $\mathrm{Z}([v_1],\ldots,[v_n]) \to \mathrm{E}$, as computed in Example \ref{ex:directed}.

\begin{corollary} \label{cor:elfisCY} Let $\pi\colon Y \to \Delta$ be an genus $1$ Lefschetz fibration. Then the asymptotic charge map $\phi\colon \mathrm{H}_2(Y,C;\bZ) \to \mathrm{H}_1(C,\bZ)$ is a relative $(-1)^0$-CY spherical homomorphism of pseudolattices.
\end{corollary}

\subsection{Quasi Landau-Ginzburg models} 

Combining Proposition \ref{prop:plisZ} with Theorem \ref{thm:aurouxclass}, we see that the problem of classifying genus $1$ Lefschetz fibrations may be reduced to a question about exceptional bases and spherical homomorphisms. In particular, we may  use Theorem \ref{thm:main} to classify genus $1$ Lefschetz fibrations, under the assumption that the asymptotic charge map is a quasi del Pezzo homomorphism. To arrange this, we make the following definition.

\begin{definition} \label{def:lg} A \emph{quasi Landau-Ginzburg \textup{(}LG\textup{)} model} is a genus $1$ Lefschetz fibration $\pi\colon Y \to \Delta$ (or simply $Y$, when no confusion is likely to result) such that, for some choice of fibre $C$ over a point in $\partial\Delta$ and basis $(a,b)$ for $\mathrm{H}_1(C,\bZ) \cong \mathrm{E}$, the following properties hold:
\begin{enumerate}
\item in the basis $(a,b)$ for $\mathrm{E}$, the total monodromy has matrix $\begin{psmallmatrix}1 & n-12 \\ 0 & 1 \end{psmallmatrix}$, where $n$ is the number of singular fibres in $Y$; and
\item $r(a)$ is primitive in $\mathrm{H}_2(Y,C;\bZ)$, where $r$ is the right adjoint to the asymptotic charge map $\phi\colon \mathrm{H}_2(Y,C;\bZ) \to \mathrm{H}_1(C,\bZ) \cong \mathrm{E}$.
\end{enumerate}
\end{definition}

\begin{remark} This definition may be thought of as a weaker version of the definition of a Landau-Ginzburg model given by Auroux, Katzarkov, and Orlov \cite{msdpsvccs}. Indeed, \cite{msdpsvccs} defines a Landau-Ginzburg model to be an elliptic fibration over $\bC$, having $n$ singular fibres of Kodaira type $\mathrm{I}_1$, that may be compactified to a rational elliptic surface over $\bP^1$ by the addition of a singular fibre of Kodaira type $\mathrm{I}_{12-n}$ at infinity. Note that this definition assumes that $n < 12$. It is easy to see that, if one shrinks the base of such a Landau-Ginzburg model to a closed disc, the resulting fibration will satisfy Definition \ref{def:lg}.
\end{remark}

 It follows that if $\pi\colon Y \to \Delta$ is a quasi LG model, then $\phi\colon \mathrm{H}_2(Y,C;\bZ) \to \mathrm{H}_1(C,\bZ)$ satisfies the conditions of Propositions \ref{prop:issurfacelike}, \ref{prop:anticanonical}, and \ref{prop:matrix}; in particular, $\mathrm{H}_2(Y,C;\bZ)$ is surface-like with point-like vector $r(a)$ and $12 - n = q(K,K) \in \bZ$, where $K$ denotes the canonical class $K = -[r(b)] \in \NS(\mathrm{H}_2(Y,C;\bZ))$. Moreover, since $n$ is the rank of $\mathrm{H}_2(Y,C;\bZ)$, Remark \ref{rem:defect} shows that the defect $\delta(\mathrm{H}_2(Y,C;\bZ)) = 0$.

By Proposition \ref{prop:plisZ} and Corollary \ref{cor:elfisCY}, we immediately see that if $\pi\colon Y \to \Delta$ is a quasi LG model, then $\phi\colon \mathrm{H}_2(Y,C;\bZ) \to \mathrm{H}_1(C,\bZ)$ satisfies conditions (1)-(3) of Definition \ref{def:qdp}. To show that condition (4) also holds, so that $\phi\colon \mathrm{H}_2(Y,C;\bZ) \to \mathrm{H}_1(C,\bZ)$ is quasi del Pezzo, we will need to work a little harder.

We begin by showing that the properties of $\mathrm{NS}(\mathrm{H}_2(Y,C;\bZ))$ may be computed from more familiar topological quantities. Recall that, for any oriented $4$-manifold $Y$, there is an intersection pairing,
\[
\mathrm{H}_2(Y,\mathbb{Z}) \times \mathrm{H}_2(Y,\mathbb{Z}) \rightarrow \mathbb{Z}
\]
making $\mathrm{H}_2(Y,\bZ)$ into a lattice. We use $\sigma(Y)$ to denote the $\sigma$-invariant of this pairing (see Definition \ref{def:signature}). The following result appears in the proof of \cite[Lemma 16]{sclf}.

\begin{lemma}
Let $\pi\colon Y \rightarrow \Delta$ be an genus $1$ Lefschetz fibration and let $\mathrm{N}$ be the kernel of the asymptotic charge map $\phi$, equipped with the pairing induced by $\langle \cdot, \cdot \rangle_{\mathrm{Sft}}$. Then 
\[
\sigma(Y) = -\sigma(\mathrm{N}).
\]
\end{lemma}

\begin{remark}The Seifert form on $\mathrm{N}$ disagrees with the form used by Auroux \cite{sclf} by a factor of $(-1)$, hence the slight difference between our statement and his.\end{remark}

Using this we obtain the following result.

\begin{lemma}\label{lem:Y=K}
Let $Y$ be a quasi LG model. Choose a fibre $C$ as in Definition \ref{def:lg} and let $K$ denote the canonical class in the N\'{e}ron-Severi lattice $\NS(\mathrm{H}_2(Y,C;\bZ))$. Let $K^\perp$ denote the orthogonal complement of $K$ in $\mathrm{NS}(\mathrm{H}_2(Y,C;\bZ))$, equipped with the pairing $q(\cdot, \cdot)$. Then $\sigma(K^\perp) = \sigma(Y)$.
\end{lemma}

\begin{proof}
For any $u \in \mathrm{H}_2(Y,C;\bZ)$, we have $\phi(u) = 0$ if and only if $\langle \phi(u), a \rangle_\mathrm{E} = \langle \phi(u), b \rangle_\mathrm{E} = 0$. Applying adjunction, we see that this is true if and only if $v \in (r(a), r(b))^\perp$, so $\mathrm{N} = \ker(\phi) = (r(a),r(b))^\perp$.   This is of course the orthogonal complement of $r(b)$ in $r(a)^\perp = \mathbf{p}^\perp$, which is equivalent to $K^\perp$ in $\mathrm{NS}(\mathrm{H}_2(Y,C;\bZ))$ by Proposition \ref{prop:anticanonical}. We thus see that $\sigma(K^\perp) = - \sigma(\mathrm{N})$, as $q$ is defined to be the negative of the Seifert form. Applying the previous lemma gives the result.
\end{proof}

The following lemma, whose proof is straightforward, shows that we may use this to deduce $\sigma(\mathrm{NS}(\mathrm{H}_2(Y,C;\bZ)))$ from $\sigma(K^{\perp}) = \sigma(Y)$ and $q(K,K)$. 

\begin{lemma}\label{lem:orth}
If $u$ is an element of a lattice $\mathrm{L}$ equipped with a nondegenerate symmetric bilinear form $q(\cdot, \cdot)$, then
\[
\sigma(u^\perp) = \begin{cases} \sigma(\mathrm{L}) -1 & \text{ if } q(u,u) > 0 \\
 \sigma(\mathrm{L}) & \text{ if } q(u,u) = 0 \\ 
\sigma(\mathrm{L}) + 1 & \text{ if } q(u,u) < 0
\end{cases}
\]
\end{lemma}

It remains to compute the invariants $\sigma(Y)$. To do this we use work of Matsumoto \cite{4mft2}, following earlier work of Meyer \cite{dsvf}. 

\begin{proposition}\label{prop:matsu}
Let $Y$ be a quasi LG model with $n$ singular fibres. Then 
\[
\sigma(Y) = \begin{cases}
5-n & \text{if } n < 12 \\
-8 & \text{if } n=12 \\
3-n & \text{if } n > 12
\end{cases}
\]
\end{proposition}

\begin{proof}
This proof is a simple minded generalization of a result of Matsumoto \cite[Theorem 6]{4mft2}. We first recall a result of Meyer \cite[Satz 5]{dsvf} which computes the $\sigma$-invariant of a torus bundle $\pi_0 \colon F \rightarrow B$ over an oriented, connected $2$-dimensional manifold with boundary $B$. For concreteness, assume that $B$ is a closed disc $\Delta$ with $n$ open discs $D_i$ removed from its interior. Decompose $\partial B = \partial \Delta \cup C_1 \cup \dots \cup C_{n}$ into its connected components and let $\tau_i,\tau_{\infty} \in \mathrm{SL}(2,\mathbb{Z})$ be the monodromies around $C_i$,  $\partial \Delta$ respectively, making sure to take each boundary curve oriented counterclockwise with respect to the orientation on $B$. Meyer \cite[\S 5]{dsvf} defines a class function 
\[
\phi \colon \mathrm{SL}(2,\mathbb{Z}) \longrightarrow (\tfrac{1}{3})\mathbb{Z},
\]
such that the value of $\phi(\tau_i)$ only depends on $\tau_i$ up to conjugation. Then \cite[Satz 5]{dsvf} states that
\begin{equation}\label{eq:mey}
\sigma(F) = \phi(\tau_{\infty}) + \sum_{i=1}^{n} \phi(\tau_i).
\end{equation}

Now let $\Sigma := \{x_1,\ldots,x_n\} \subset \Delta$ be the critical set of $\pi\colon Y \to \Delta$ and let $F_i$ be the singular fibre over $x_i$. If we let $D_i$ be small open discs around each $x_i$ and let $N_i := \pi^{-1}(D_i)$, then $F := Y \setminus \bigcup_{i=1}^n N_i$ is a torus fibration over $B := \Delta \setminus \bigcup_{i=1}^n D_i$, to which we may apply Equation \eqref{eq:mey} to compute the signature.

Once this is done, we may apply Novikov's additivity theorem or, more precisely, Wall's non-additivity theorem \cite{nas}, to show that
\[
\sigma(Y) = \sigma(F) + \sum_{i=1}^n \sigma(N_i) = \phi(\tau_{\infty}) + \sum_{i=1}^n(\phi(\tau_i) + \sigma(N_i)).
\]

Matsumoto \cite[Corollary 8.1]{4mft2} shows that for a singular fibre obtained by collapsing a simple closed curve on a smooth $2$-torus (Kodaira type $\mathrm{I}_1$), $\phi(\tau_i) + \sigma(N_i) = -\frac{2}{3}$. Therefore, since all of our fibres are of this type, we have that
\[
\sigma(Y) = \phi\begin{psmallmatrix} 1 & n-12 \\ 0 & 1 \end{psmallmatrix}  - \tfrac{2n}{3}
\]
According to \cite[pp. 261]{dsvf}, for $d \in \bZ$ we have that
\[
\phi\begin{pmatrix} 1 & d \\ 0 & 1 \end{pmatrix} = \begin{cases}
 1 - d/3 & \text{ if } d < 0 \\
 0 & \text{ if } d = 0 \\
 -1 - d/3 & \text{ if } d > 0
\end{cases}
\]
Putting this together with the equation for $\sigma(Y)$ above proves the claim.
\end{proof}

 We now combine these results to prove that condition (4) from Definition \ref{def:qdp} holds.

\begin{proposition}\label{prop:signature}
If $Y$ is a quasi LG model with $n$ singular fibres and chosen fibre $C$, then the signature of the N\'{e}ron-Severi lattice $\mathrm{NS}(\mathrm{H}_2(Y,C;\bZ))$ is $(1,n-3)$.
\end{proposition}

\begin{proof}
We know by Proposition \ref{prop:matrix} that $q(K, K) = 12-n$. Therefore, by Lemmas \ref{lem:Y=K} and \ref{lem:orth}, we see that 
\[
\sigma(\NS(\mathrm{H}_2(Y,C;\bZ))) = \begin{cases}\sigma(Y) - 1 & \text{ if } n < 12 \\
 \sigma(Y) & \text{ if } n = 12\\
\sigma(Y) - 1 & \text{ if } n > 12
\end{cases}
\]
Hence, by Proposition \ref{prop:matsu}, $\sigma(\NS(\mathrm{H}_2(Y,C;\bZ))) = 4-n$. Since $\NS(\mathrm{H}_2(Y,C;\bZ))$ is a nondegenerate lattice of rank $n$, it follows that its signature is $(1,n-3)$, as required.
\end{proof}

Putting everything together, we obtain the following result.

\begin{theorem} \label{thm:Lefschetzisqdp} Let $\pi\colon Y \to \Delta$ be a quasi Landau-Ginzburg model. Choose a fibre $C$ over a point in $\partial \Delta$ and a basis $(a,b)$ for $\mathrm{H}_1(C,\bZ) \cong \mathrm{E}$ so that Definition \ref{def:lg} holds. Then the asymptotic charge map $\phi \colon \mathrm{H}_2(Y,C;\bZ) \to \mathrm{H}_1(C,\bZ)$ is a quasi del Pezzo homomorphism.
\end{theorem}

We may therefore apply Theorem \ref{thm:main} in this setting. 

\begin{theorem} \label{thm:lgclass} Let $\pi\colon Y \to \Delta$ be a quasi Landau-Ginzburg model with $n$ singular fibres. Choose a fibre $C$ over a point in $\partial \Delta$ and a basis $(a,b)$ for $\mathrm{H}_1(C,\bZ) \cong \mathrm{E}$ so that Definition \ref{def:lg} holds. Then, after a sequence of Hurwitz moves and a global conjugation \textup{(}which preserves the total monodromy\textup{)}, $\pi\colon Y \to \Delta$ has monodromy factorization
\begin{align*}
\begin{psmallmatrix}1 & n-12 \\ 0 & 1 \end{psmallmatrix} &= M_{0,1}M_{3,1}M_{6,1}M_{1,0}^{n-3} = \begin{psmallmatrix}1 & 0 \\-1 & 1\end{psmallmatrix}\begin{psmallmatrix}-2 & 9 \\-1 & 4\end{psmallmatrix}\begin{psmallmatrix}-5 & 36 \\-1 & 7\end{psmallmatrix}\begin{psmallmatrix}1 & 1 \\0 & 1\end{psmallmatrix}^{n-3},\\
\begin{psmallmatrix}1 & -8 \\ 0 & 1 \end{psmallmatrix} &= M_{0,1}M_{2,1}M_{2,1}M_{4,1} = \begin{psmallmatrix}1 & 0 \\-1 & 1\end{psmallmatrix}\begin{psmallmatrix}-1 & 4 \\-1 & 3\end{psmallmatrix}\begin{psmallmatrix}-1 & 4 \\-1 & 3\end{psmallmatrix}\begin{psmallmatrix}-3 & 16 \\-1 & 5\end{psmallmatrix},
\end{align*}
when $n \geq 3$ and $n = 4$ respectively. Moreover, the two cases with $n=4$ are not equivalent up to Hurwitz moves and global conjugation.
\end{theorem}
\begin{proof} This is an immediate consequence of Theorem \ref{thm:main}, Proposition \ref{prop:plisZ}, and the explicit description of the Dehn twist matrices $M_{p,q}$.
\end{proof}

\begin{remark} The existence of two distinct monodromy factorizations when $n = 4$ was previously observed by Auroux \cite[Proposition 3.1]{fsseisf}, who also showed that the corresponding genus $1$ Lefschetz fibrations $\pi\colon Y \to \Delta$ are topologically distinct.
\end{remark}

The following corollary is an immediate consequence of Theorems \ref{thm:aurouxclass} and \ref{thm:lgclass}.

\begin{corollary} \label{cor:lgclass}  There is precisely one quasi Landau-Ginzburg model with $n$ singular fibres for each $n \geq 3$ and $n \neq 4$, up to diffeomorphism, and there are two when $n = 4$. 
\end{corollary}

\begin{remark} It is interesting to ask what happens if we relax the assumption on the relationship between the number of singular fibres of $Y$ and the total monodromy in Definition \ref{def:lg}. Indeed, suppose that $Y$ contains $n$ singular fibres and that the total monodromy has matrix $\begin{psmallmatrix}1 & -d \\ 0 & 1 \end{psmallmatrix}$, for some $d$. In this case, the defect $\delta(\mathrm{H}_2(Y,C;\bZ)) = d + n - 12$; we call such fibrations \emph{defective Landau-Ginzburg models}.

The classification of defective Landau-Ginzburg models poses an interesting question. As a first step, it is not difficult to show that the defect must be divisible by $12$. Indeed, as the abelianization of $\mathrm{SL}(2,\bZ)$ is $\bZ/12\bZ$ and Dehn twists are sent to $1 \in \bZ/12\bZ$ under this map, any factorization of the identity matrix $I_{2\times 2}$ into a product of $k$ Dehn twist matrices $M_{p,q}$ has $k$ divisible by $12$. Moreover, the monodromy factorization gives a factorization of $\begin{psmallmatrix}1 & -d \\ 0 & 1 \end{psmallmatrix}$ into a product of $n$ such matrices. 

If $d \geq 0$, we may write $I_{2\times 2} =  \begin{psmallmatrix}1 & -d \\ 0 & 1 \end{psmallmatrix} M_{1,0}^d$, so we must have $n+d$ divisible by $12$ in this case. On the other hand, if $d <0$, note that we may write $\begin{psmallmatrix}1 & -9 \\ 0 & 1 \end{psmallmatrix} = M_{0,1}M_{3,1}M_{6,1}$. Set $d = -9q + r$ for $q > 0$ and $0 \leq r < 9$. Then we may write 
\[\begin{psmallmatrix}1 & -d \\ 0 & 1 \end{psmallmatrix}\begin{psmallmatrix}1 & -9 \\ 0 & 1 \end{psmallmatrix}^qM_{1,0}^r = \begin{psmallmatrix}1 & -d \\ 0 & 1 \end{psmallmatrix}(M_{0,1}M_{3,1}M_{6,1})^qM_{1,0}^r = I_{2\times 2},\]
so $n + 3q + r \equiv n + d \equiv 0 \pmod {12}$. Thus the defect is divisible by $12$ in all cases.

When $d \leq 0$, it follows from Theorem \ref{thm:aurouxclass} and a result of Cadavid and V\'{e}lez \cite[Theorem 20]{nfscsfef} that a defective Landau-Ginzburg model is uniquely determined, up to diffeomorphism, by $n$ and $d$, and all cases are realized subject to the condition that the defect is divisible by $12$. However, in general the question of classifying defective Landau-Ginzburg models seems to be open. The obstacle to using our methods in this more general setting is the fact that defective Landau-Ginzburg models fail the signature condition (4) in Definition \ref{def:qdp}, which prevents us from applying Kuznetsov's theorem (Theorem \ref{thm:kuz}); an analogue of Kuznetsov's theorem without this signature condition seems difficult to prove.
\end{remark}

\section{Mirror Symmetry} \label{sec:mirrors}

The fact that the classification of quasi Landau-Ginzburg models parallels the classification of quasi del Pezzo surfaces may initially seem surprising, but in the context of mirror symmetry this result is very natural. Indeed, homological mirror symmetry predicts that if $Z$ and $\check{Z}$ are mirror objects, then there is a an algebraic category $\mathbf{D}(Z)$, whose objects are certain complexes of coherent sheaves on $Z$, which is derived equivalent to a symplectic category $\mathbf{F}(\check{Z})$, whose objects are built from decorated Lagrangians on $\check{Z}$.

For example, Polishchuk, Zaslow, Abouzaid, and Smith \cite{cmsec,hms4t} have proved that, if $C$ and $\check{C}$ are dual elliptic curves, then there is a derived equivalence between the bounded derived category $\mathbf{D}(C)$ of coherent sheaves on $C$ and the derived Fukaya category $\mathbf{F}(\check{C})$ of $\check{C}$. In particular, we find that the numerical Grothendieck groups $\mathrm{K}_0^{\mathrm{num}}(\mathbf{D}(C))$ and $\mathrm{K}_0^{\mathrm{num}}(\mathbf{F}(\check{C})) \cong \mathrm{H}_1(\check{C},\bZ)$ are isomorphic, which explains the fact (see Example \ref{ex:elliptic}) that both are isomorphic to the elliptic curve pseudolattice $\mathrm{E}$.

In our setting, given the results of Auroux, Katzarkov, and Orlov \cite{msdpsvccs} on homological mirror symmetry for del Pezzo surfaces, it does not seem unreasonable to postulate that the (homological) mirror to a  quasi del Pezzo surface $X$ with degree $K_X^2 = d$ should be a quasi Landau-Ginzburg model $\pi\colon Y \to \Delta$ with $12-d$ singular fibres. If this were the case, we would expect there to be a derived equivalence between the bounded derived category $\mathbf{D}(X)$ of coherent sheaves on $X$ and the \emph{derived Fukaya-Seidel category}  $\mathbf{F}(Y)$. Whilst we do not prove such a statement, we can exhibit its shadow on the level of numerical Grothendieck groups.

Indeed, if $Y$ is a quasi LG model with chosen fibre $C$, it follows from the discussion in \cite[Section 6.1]{fscslf} that  $\mathrm{K}_0^{\mathrm{num}}(\mathbf{F}(Y))$, with its Euler pairing, is isomorphic as a pseudolattice to $\mathrm{H}_2(Y,C;\bZ)$ with the Seifert pairing. Using this fact, the main results of this paper immediately give the following theorem, which may be thought of as a special case of a conjecture of Auroux \cite[Conjecture 7.7]{mstdcad}.

\begin{theorem} \label{thm:mirror} Let $X$ be a quasi del Pezzo surface, with smooth anticanonical divisor $C$ satisfying $C^2 = d$. Then there exists a quasi Landau-Ginzburg model $\pi\colon Y \to \Delta$ with $12-d$ singular fibres and distinguished fibre $\check{C}$, which is unique up to diffeomorphism, such that there are isomorphisms of pseudolattices making the following diagram commute
\[\xymatrix{ \mathrm{K}_0^{\mathrm{num}}(\mathbf{D}(X)) \ar[r]^{\sim} \ar[d]^{i^*}& \mathrm{K}_0^{\mathrm{num}}(\mathbf{F}(Y)) \ar[d]^{\phi} \\ \mathrm{K}_0^{\mathrm{num}}(\mathbf{D}(C)) \ar[r]^{\sim} & \mathrm{K}_0^{\mathrm{num}}(\mathbf{F}(\check{C})).
}\]
Here $i^*\colon \mathrm{K}_0^{\mathrm{num}}(\mathbf{D}(X)) \to \mathrm{K}_0^{\mathrm{num}}(\mathbf{D}(C))$ is the map induced by the inclusion $i\colon C \to X$ and $\phi$ is the asymptotic charge map.
\end{theorem}

 It can also be shown that Theorem \ref{thm:mirror} gives a one-to-one correspondence between complex deformation classes of quasi del Pezzo surfaces and diffeomorphism classes of quasi Landau-Ginzburg models, as one would expect from a mirror statement. In fact, mirror symmetry suggests that one should be able to go further, and relate the complex moduli of quasi del Pezzo surfaces to the space of symplectic structures on the corresponding quasi Landau-Ginzburg models, but this question is beyond the scope of this paper.
 \medskip
 
 We conclude with a brief note about noncommutative geometry. In \cite{mswppnd}, Auroux, Katzarkov, and Orlov discuss homological mirror symmetry for noncommutative projective planes. On the level of pseudolattices, however, there is no difference between the commutative and noncommutative settings in this case: there is only one surface-like pseudolattice of rank $3$ (c.f. Proposition \ref{prop:3classification}).
 
 If we proceed to pseudolattices of rank $4$, however, this situation changes. The classification of rank $4$ surface-like pseudolattices by de Thanhoffer de Volcsey and Van den Bergh \cite[Theorem A]{ameecl4} shows that there are infinitely many such pseudolattices: an infinite family parametrized by $n \in \bN$ plus a single exceptional case. Only two of these pseudolattices arise from commutative del Pezzo surfaces: the Hirzebruch surface $\bF_1$ (corresponding to $n = 1$ in the infinite family) and $\bP^1 \times \bP^1$ (corresponding to the exceptional case). This is essentially the content of Proposition \ref{prop:4classification}. 
 
However, results of Belmans and Presotto \cite{cncsecl4} suggest that the remaining cases $n \geq 2$ can be realized by noncommutative surfaces, obtained by constructing sheaves of maximal orders on $\bF_1$; intuitively, this construction may be thought of as blowing up $\mathbb{P}^2$ in a fat point. Interestingly, \cite[Proposition 39]{cncsecl4} shows that these orders are only del Pezzo, in the sense of \cite{dpops}, when $n = 2$.
 
Examining the proof of Proposition \ref{prop:4classification}, one may check that most of the conditions defining a quasi del Pezzo homomorphism still hold in this noncommutative setting (although, considering the discussion above, perhaps ``del Pezzo'' is not an appropriate choice of terminology in this setting); the only exception is the primitivity condition \ref{def:qdp}(3). Intuitively, this happens because the blow-ups considered in \cite{cncsecl4} occur at fat points, which correspond to imprimitive elements of the elliptic curve pseudolattice. This suggests that it may be possible to extend our definition of quasi del Pezzo homomorphisms to the noncommutative setting by dropping the primitivity condition  \ref{def:qdp}(3).

In the context of \cite{msdpsvccs,mswppnd}, it should also be possible to see how these ``fat point blow-ups'' behave under mirror symmetry for noncommutative surfaces. Na\"{i}vely, we might expect a non-commutative Landau-Ginzburg model to look like a genus $1$ Lefschetz fibration that contains some ``fat vanishing cycles'' (i.e. vanishing cycles that are not primitive), as this would give the required relaxing of the primitivity condition \ref{def:qdp}(3). However, it is unclear how such objects should be realized geometrically.

\bibliography{publications,preprints}

\providecommand{\bysame}{\leavevmode\hbox to3em{\hrulefill}\thinspace}
\providecommand{\MR}{\relax\ifhmode\unskip\space\fi MR }
% \MRhref is called by the amsart/book/proc definition of \MR.
\providecommand{\MRhref}[2]{%
  \href{http://www.ams.org/mathscinet-getitem?mr=#1}{#2}
}
\providecommand{\href}[2]{#2}
\begin{thebibliography}{dTdVdB16}

\bibitem[AGZV88]{sdm2}
V.~I. Arnold, S.~M. Guse{\u\i}n-Zade, and A.~N. Varchenko, \emph{Singularities
  of differentiable maps. {V}ol. {II}, monodromy and asymptotics of integrals},
  Monographs in Mathematics, vol.~83, Birkh{\"{a}}user, Boston, 1988.

\bibitem[AKO06]{msdpsvccs}
D.~Auroux, L.~Katzarkov, and D.~Orlov, \emph{Mirror symmetry for del {P}ezzo
  surfaces: Vanishing cycles and coherent sheaves}, Invent. Math. \textbf{166}
  (2006), no.~3, 537--582.

\bibitem[AKO08]{mswppnd}
\bysame, \emph{Mirror symmetry for weighted projective planes and their
  noncommutative deformations}, Ann. of Math. (2) \textbf{167} (2008), no.~3,
  867--943.

\bibitem[AL17]{sdgf}
R.~Anno and T.~Logvinenko, \emph{Spherical {DG}-functors}, J. Eur. Math. Soc.
  (JEMS) \textbf{19} (2017), no.~9, 2577--2656.

\bibitem[AS10]{hms4t}
M.~Abouzaid and I.~Smith, \emph{Homological mirror symmetry for the 4-torus},
  Duke Math. J. \textbf{152} (2010), no.~3, 373--440.

\bibitem[Aur05]{sclf}
D.~Auroux, \emph{A stable classification of {L}efschetz fibrations}, Geom.
  Topol. \textbf{9} (2005), 203--217.

\bibitem[Aur07]{mstdcad}
\bysame, \emph{Mirror symmetry and {T}-duality in the complement of an
  anticanonical divisor}, J. G{\"{o}}kova Geom. Topol. \textbf{1} (2007),
  51--91.

\bibitem[Aur15]{fsseisf}
\bysame, \emph{Factorizations in {$SL(2,\mathbb{Z})$} and simple examples of
  inequivalent {S}tein fillings}, J. Symplectic Geom. \textbf{13} (2015),
  no.~2, 261--277.

\bibitem[BK15]{cblfsfg}
R.~{\.{I}}. Baykur and S.~Kamada, \emph{Classification of broken {L}efschetz
  fibrations with small fiber genera}, J. Math. Soc. Japan \textbf{67} (2015),
  no.~3, 877--901.

\bibitem[Bon04]{sgtbfbg}
A.~I. Bondal, \emph{A symplectic groupoid of triangular bilinear forms and the
  braid group}, Izv. Ross. Akad. Nauk Ser. Mat. \textbf{68} (2004), no.~4,
  19--74.

\bibitem[BP93]{hpaamh}
A.~I. Bondal and A.~E. Polishchuk, \emph{Homological properties of associative
  algebras: the method of helices}, Izv. Ross. Akad. Nauk Ser. Mat. \textbf{57}
  (1993), no.~2, 3--50.

\bibitem[BP18]{cncsecl4}
P.~Belmans and D.~Presotto, \emph{Construction of non-commutative surfaces with
  exceptional collections of length 4}, J. Lond. Math. Soc. (2) \textbf{98}
  (2018), no.~1, 85--103.

\bibitem[CK03]{dpops}
D.~Chan and R.~S. Kulkarni, \emph{{d}el {P}ezzo orders on projective surfaces},
  Adv. Math. \textbf{173} (2003), no.~1, 144--177.

\bibitem[CT88]{asdps}
D.~F. Coray and M.~A. Tsfasman, \emph{Arithmetic on singular del {P}ezzo
  surfaces}, Proc. London Math. Soc. (3) \textbf{57} (1988), no.~1, 25--87.

\bibitem[CV09]{nfscsfef}
C.~A. Cadavid and J.~D. V{\'{e}}lez, \emph{Normal factorization in
  {$\mathrm{SL}(2,\mathbb{Z})$} and the confluence of singular fibers in
  elliptic fibrations}, Beitr{\"{a}}ge Algebra Geom. \textbf{50} (2009), no.~2,
  405--423.

\bibitem[Dem80]{sdp}
M.~Demazure, \emph{Surfaces de del {P}ezzo}, S{\'{e}}minaire sur les
  Singularit{\'{e}}s des Surfaces (M.~Demazure, H.~Pinkham, and B.~Teissier,
  eds.), Lecture Notes in Math., vol. 777, Springer, 1980, pp.~21--69.

\bibitem[dTdVdB16]{ameecl4}
L.~de~Thanhoffer~de Volcsey and M.~Van den Bergh, \emph{On an analogue of the
  {M}arkov equation for exceptional collections of length 4}, Preprint, July
  2016, \texttt{arXiv:1607.04246}.

\bibitem[DZ99]{sjalar}
O.~DeWolfe and B.~Zwiebach, \emph{String junctions for arbitrary {L}ie-algebra
  representations}, Nuclear Phys. B \textbf{541} (1999), no.~3, 509--565.

\bibitem[EL16]{feclbdps}
A.~Elagin and V.~Lunts, \emph{On full exceptional collections of line bundles
  on del {P}ezzo surfaces}, Mosc. Math. J. \textbf{16} (2016), no.~4, 691--709.

\bibitem[FYY00]{mwlsj}
M.~Fukae, Y.~Yamada, and S.-K. Yang, \emph{Mordell-{W}eil lattice via string
  junctions}, Nuclear Phys. B \textbf{572} (2000), no.~1-2, 71--94.

\bibitem[GGM85]{dmfpsscn}
A.~Galligo, M.~Granger, and P.~Maisonobe, \emph{{${\mathscr{D}}$}-modules et
  faisceaux pervers dont le support singulier est un croisement normal}, Ann.
  Inst. Fourier (Grenoble) \textbf{35} (1985), no.~1, 1--48.

\bibitem[GHS16]{gtsj}
A.~Grassi, J.~Halverson, and J.~L. Shaneson, \emph{Geometry and topology of
  string junctions}, J. Singul. \textbf{15} (2016), 36--52.

\bibitem[GL16]{sfctb}
M.~Golla and P.~Lisca, \emph{On {S}tein fillings of contact torus bundles},
  Bull. Lond. Math. Soc. \textbf{48} (2016), no.~1, 19--37.

\bibitem[GS99]{4mkc}
R.~E. Gompf and A.~I. Stipsicz, \emph{{$4$}-manifolds and {K}irby calculus},
  Graduate Studies in Mathematics, vol.~20, American Mathematical Society,
  1999.

\bibitem[KS14]{ps}
M.~Kapranov and V.~Schechtman, \emph{Perverse schobers}, Preprint, November
  2014, \texttt{arXiv:1411.2772}.

\bibitem[KS16]{psgs}
\bysame, \emph{Perverse sheaves and graphs on surfaces}, Preprint, January
  2016, \texttt{arXiv:1601.01789}.

\bibitem[Kuz17]{ecslc}
A.~G. Kuznetsov, \emph{Exceptional collections in surface-like categories}, Sb.
  Math. \textbf{208} (2017), no.~9, 1368--1398.

\bibitem[Mar80]{fqbi}
A.~Markoff, \emph{Sur les formes quadratiques binaires ind{\'{e}}finies}, Math.
  Ann. \textbf{17} (1880), no.~3, 379--399.

\bibitem[Mat83]{4mft2}
Y.~Matsumoto, \emph{On {$4$}-manifolds fibered by tori. {II}}, Proc. Japan
  Acad. Ser. A Math. Sci. \textbf{59} (1983), no.~3, 100--103.

\bibitem[Mey73]{dsvf}
W.~Meyer, \emph{Die {S}ignatur von {F}l{\"{a}}chenb{\"{u}}ndeln}, Math. Ann.
  \textbf{201} (1973), 239--264.

\bibitem[Moi77]{cscscpp}
B.~Moishezon, \emph{Complex surfaces and connected sums of complex projective
  planes}, Lecture Notes in Math., vol. 603, Springer-Verlag, 1977.

\bibitem[Orl92]{pbmtdccs}
D.~Orlov, \emph{Projective bundles, monoidal transformations, and derived
  categories of coherent sheaves}, Izv. Ross. Akad. Nauk Ser. Mat. \textbf{56}
  (1992), no.~4, 852--862.

\bibitem[Per90]{ckfres}
U.~Persson, \emph{Configurations of {K}odaira fibers on rational elliptic
  surfaces}, Math. Z. \textbf{205} (1990), no.~1, 1--47.

\bibitem[Per18]{caesrs}
M.~Perling, \emph{Combinatorial aspects of exceptional sequences on (rational)
  surfaces}, Math. Z. \textbf{288} (2018), no.~1-2, 243--286.

\bibitem[PZ98]{cmsec}
A.~Polishchuk and E.~Zaslow, \emph{Categorical mirror symmetry: the elliptic
  curve}, Adv. Theor. Math. Phys. \textbf{2} (1998), no.~2, 443--470.

\bibitem[Sug16]{fscslf}
S.~Sugiyama, \emph{On the {F}ukaya-{S}eidel categories of surface {L}efschetz
  fibrations}, Preprint, July 2016, \texttt{arXiv:1607.02263}.

\bibitem[Via17]{ecnsls}
C.~Vial, \emph{Exceptional collections, and the {N}{\'{e}}ron-{S}everi lattice
  for surfaces}, Adv. Math. \textbf{305} (2017), 895--934.

\bibitem[Wal69]{nas}
C.~T.~C. Wall, \emph{Non-additivity of the signature}, Invent. Math. \textbf{7}
  (1969), 269--274.

\end{thebibliography}
\bibliographystyle{amsalpha}
\end{document}